\documentclass[12pt,fleqn,oneside]{article}
\usepackage[inner=32mm,outer=31mm,tmargin=30mm,bmargin=38mm]{geometry}
\usepackage[utf8]{inputenc}
\usepackage[T1]{fontenc}
\usepackage{amssymb}
\usepackage{hyperref}
\usepackage{geometry}
\usepackage{amsmath}
\usepackage{amsthm}
\usepackage{xfrac}
\usepackage{graphicx}
\usepackage{color}
\usepackage[normalem]{ulem}
\usepackage[english]{babel}
\usepackage[onehalfspacing]{setspace}
\usepackage{verbatim}

\usepackage{rotating}

\newcommand{\IR}{\ensuremath{\mathbb{R}}}
\newcommand{\IN}{\ensuremath{\mathbb{N}}}
\newcommand{\IZ}{\ensuremath{\mathbb{Z}}}
\newcommand{\IQ}{\ensuremath{\mathbb{Q}}}

\newcommand{\IE}{\ensuremath{\mathbb{E}}}
\newcommand{\Var}{\mathrm{Var}}

\newcommand{\Hmix}{\ensuremath{H^{r,{\rm mix}}(\IR^d)}}
\newcommand{\Hmixunit}{\ensuremath{{H^{r,{\rm mix}}([0,1]^d)}}}
\newcommand{\Hmixcomp}{\ensuremath{\mathring{H}^{r,{\rm mix}}(\IR^d)}}
\newcommand{\Hmixunitcomp}{\ensuremath{\mathring{H}^{r,{\rm mix}}([0,1]^d)}}
\newcommand{\Hiso}{\ensuremath{H^s(\IR^d)}}
\newcommand{\Hisounit}{\ensuremath{{H^s([0,1]^d)}}}
\newcommand{\Hisocomp}{\ensuremath{\mathring{H}^s(\IR^d)}}
\newcommand{\Hisounitcomp}{\ensuremath{\mathring{H}^s([0,1]^d)}}
\newcommand{\Cc}{\ensuremath{C_c(\IR^d)}}

\newcommand{\supp}{\mathop{\mathrm{supp}}}
\newcommand{\diag}{\mathop{\mathrm{diag}}}
\newcommand{\mixnorm}[1]{\left\Vert#1\right\Vert_{{H^{r,{\rm mix}}(\IR^d)}}} 
\newcommand{\mixscalar}[2]{{\left\langle#1,#2\right\rangle_{H^{r,{\rm mix}}(\IR^d)}}}
\newcommand{\isonorm}[1]{\left\Vert#1\right\Vert_{{H^s}(\IR^d)}}
\newcommand{\isoscalar}[2]{\left\langle#1,#2\right\rangle_{{H^s(\IR^d)}}}
\newcommand{\lnorm}[1]{\left\Vert#1\right\Vert_{L^2(\IR^d)}}
\newcommand{\lnormunit}[1]{\left\Vert#1\right\Vert_{L^2([0,1]^d)}}
\newcommand{\lscalar}[2]{\left\langle#1,#2\right\rangle_{L^2(\IR^d)}}
\newcommand{\lscalarunit}[2]{\left\langle#1,#2\right\rangle_{L^2([0,1]^d)}}
\newcommand{\set}[1]{\left\{#1\right\}}
\newcommand{\abs}[1]{\left|#1\right|}
\newcommand{\braces}[1]{\left(#1\right)}
\renewcommand{\d}{{\rm d}} 
\newcommand{\scalar}[2]{\left\langle#1,#2\right\rangle}
\newcommand{\mixnormunit}[1]{\left\Vert#1\right\Vert_{{H^{r,{\rm mix}}([0,1]^d)}}} 
\newcommand{\mixscalarunit}[2]{{\left\langle#1,#2\right\rangle_{H^{r,{\rm mix}}([0,1]^d)}}}
\newcommand{\isonormunit}[1]{\left\Vert#1\right\Vert_{{H^s}([0,1]^d)}}
\newcommand{\isoscalarunit}[2]{\left\langle#1,#2\right\rangle_{{H^s([0,1]^d)}}}

\newcommand{\diff}{D} 

\theoremstyle{break}
\newtheorem{thm}{Theorem}
\newtheorem{prop}{Proposition}

\newtheorem{cor}{Corollary}
\theoremstyle{plain}
\newtheorem{lemma}{Lemma}

\usepackage{etoolbox}
\patchcmd{\abstract}{small}{}{}{}

\title{On the Randomization of Frolov's Algorithm for Multivariate Integration} 
 
\author{David Krieg}

\begin{document}

\thispagestyle{empty}

\begin{minipage}{0.99\linewidth}
\begin{minipage}[t][33mm]{0.2\linewidth}
\includegraphics[width=19mm]{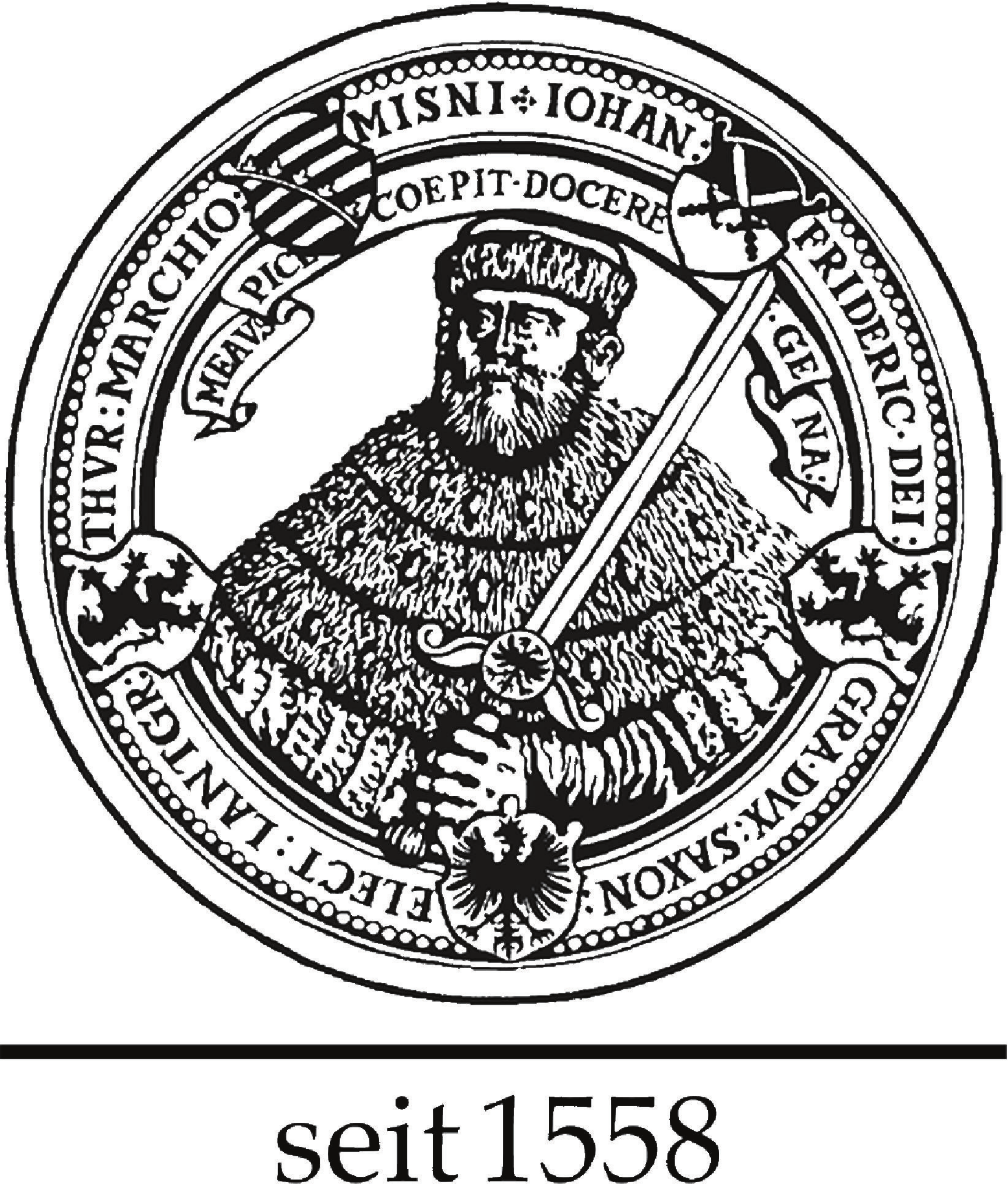} 
\end{minipage}
\begin{minipage}[b]{0.8\linewidth}
\hspace{4.1cm}{\large Friedrich-Schiller-Universit\"at Jena}\\
\hspace*{2cm} \rule[6.2pt]{2cm}{0.2mm}\rule[6.2pt]{7.5cm}{1mm}
\end{minipage}
\end{minipage}

\vspace*{13mm}
\large

\begin{center}
{\bf \Large On the Randomization of Frolov's Algorithm for Multivariate Integration}

\vspace*{10mm}

M\,A\,S\,T\,E\,R\,A\,R\,B\,E\,I\,T

\vspace*{3mm}

zur Erlangung des akademischen Grades

\vspace*{3mm}

Master of Science (M.\,Sc.)

\vspace*{3mm}

im Studiengang Mathematik

\vspace*{10mm}

FRIEDRICH-SCHILLER-UNIVERSITÄT JENA

\vspace*{3mm}

Fakultät für Mathematik und Informatik

\vspace*{10mm}

eingereicht von David Krieg

\vspace*{3mm}

geboren am 08.07.1991 in Würzburg

\vspace*{10mm}

Betreuer: Prof.\,Dr.\,Erich Novak

\vspace*{14mm}

Jena, 04.02.2016 
\end{center}

\thispagestyle{empty}
\cleardoublepage
\normalsize
\thispagestyle{empty}

\begin{abstract}
We are concerned with the numerical integration of functions from 
the Sobolev space $\Hmixunit$ of dominating mixed smoothness
$r\in\IN$ over the $d$-dimensional unit cube.

In \cite{frolov}, K.\,K.\,Frolov introduced a deterministic quadrature
rule whose worst case error has the order $n^{-r} \, (\log n)^{(d-1)/2}$
with respect to the number $n$ of function evaluations.
This is known to be optimal.
In \cite{krinov}, 39 years later, Erich Novak and me introduced a
randomized version of this algorithm using $d$ random dilations.
We showed that its error is bounded above by a constant multiple of
$n^{-r-1/2} \, (\log n)^{(d-1)/2}$
in expectation and by $n^{-r} \, (\log n)^{(d-1)/2}$ almost surely.
The main term $n^{-r-1/2}$ is again optimal and it turns out
that the very same algorithm is also optimal for the isotropic Sobolev space
$\Hisounit$ of smoothness $s>d/2$.
We also added a random shift to this algorithm to make it unbiased.
Just recently, Mario Ullrich proved in \cite{ullrichneu} that the expected
error of the resulting algorithm on $\Hmixunit$ is even bounded above by $n^{-r-1/2}$.
This thesis is a review of the mentioned upper bounds and their proofs.
\end{abstract}

\newpage

\tableofcontents

\newpage
\section{Introduction}

Many applications deal with multivariate functions $f$ which are smooth
in the sense that certain weak derivatives $\diff^\alpha f$ exist and are
square-integrable, functions from a \textit{Sobolev space}.

Which derivatives $\diff^\alpha f$ of $f$ are known to be existent and
square-integrable highly depends on the actual problem.
Classically, $\alpha$ covers the range of all vectors
in $\IN_0^d$ with $\Vert\alpha\Vert_1\leq s$ for some $s\in\IN$.
The corresponding Sobolev space is called \textit{isotropic Sobolev space of smoothness $s\in\IN$}.
For instance, the solutions of elliptic partial differential
equations in general and Poisson's equation in particular, have this form.
They typically appear in electrostatics or continuum mechanics.

But often $f$ is known to satisfy a stronger smoothness condition:
Derivatives $\diff^\alpha f$ for each $\alpha\in\IN_0^d$
with $\Vert\alpha\Vert_\infty\leq s$ exist and are square-integrable.
This is typically the case, if $f$ is a tensor product of $s$-times 
differentiable functions of one variable: $f(x_1,\dots,x_d)=f_1(x_1)\cdot\hdots\cdot f_d(x_d)$.
We say that $f$ is from a \textit{Sobolev space of dominating mixed smoothness $s$}.
For example, solutions of the electronic Schrödinger equation are of this
form.\\

We are concerned with the numerical integration of such functions and refer to
\cite{hartri} and \cite{giltru} for a treatise on elliptic partial
differential equations and their connection with Sobolev spaces
and to \cite{yserentant} for further information about
electronic wave functions.

More precisely,
we want to use linear quadrature rules to approximate the integral $I_d(f)$ of integrable,
real valued functions $f$ in $d$ real variables, with a particular interest
in functions with dominating mixed smoothness $s$.
A linear \textit{quadrature rule}, \textit{algorithm} or \textit{method} $A_n$
is given by a finite number $n$ of weights $a_1,\dots,a_n\in\IR$ and nodes
$x^{(1)},\dots,x^{(n)}\in\IR^d$, and the rule
\[
A_n(f)=\sum\limits_{j=1}^{n} a_j\, f\braces{x^{(j)}}
.\]
All these numbers and vectors can be deterministic or random variables.
Since $n$ counts the number of function values computed by $A_n$,
it is a measure for the cost of $A_n$, commonly
referred to as \textit{information cost} of the algorithm.

The error of $A_n$ associated with the integration of $f$ is
$\abs{A_n(f)-I_d(f)}$.
We are interested in sequences $\braces{A_n}_{n\in\IN}$ of quadrature rules
whose error decreases fast with respect to growing information cost $n$.
In this sense, numerical integration of functions 
with dominating mixed smoothness $s$ is significantly easier
than the integration of functions with isotropic smoothness $s$,
especially if the number $d$ of variables is large:
It turns out that the convergence order $n^{-s-1/2}$ can be achieved for the expected error,
while $n^{-s/d-1/2}$ is the best possible rate in the isotropic case.\\

From now on, for the sake of distinction, we will use $s$ as a parameter for isotropic
smoothness and $r$ as a parameter for dominating mixed smoothness.
The smoothness parameters $r$ and $s$ and the dimension $d$
are arbitrary natural numbers, with the single condition that $s>d/2$.
But they are considered to be fixed in the sense that
any constant in this thesis is merely a constant with respect to the
information cost $n$ and may depend on $r,s$ and $d$.\\

Let us end this introductory section with an outline of the thesis.

We start with a brief compilation of the definitions and fundamental properties
of the above mentioned Sobolev spaces.
In Section~\ref{basicquadrulesection}, we will present a familiy of deterministic quadrature rules
for the integration of compactly supported, continuous functions.
Among those rules is Frolov's algorithm, which will be examined in
Section \ref{frolovsrulesection}.
With respect to the information cost $n$, its integration error for functions $f$ with
dominating mixed smoothness $r$ and compact support in the open unit cube $(0,1)^d$
is bounded above by a constant multiple of
$n^{-r} \, (\log n)^{(d-1)/2}$ times the corresponding norm of $f$.
The order $n^{-r} \, (\log n)^{(d-1)/2}$ is optimal.
For functions with support in $(0,1)^d$ and isotropic smoothness $s$
the order $n^{-s/d}$ is achieved, which is also optimal.

In Section~\ref{randomdilationsection}, we will add random dilations to Frolov's algorithm
and examine the integration error of the resulting algorithm for the same types of functions.
We will see that in both cases the random dilations improve the order of
the algorithm's error by $1/2$ in expectation, while
not changing it in the worst case.
The additional random shift introduced in Section \ref{randomshiftsection} 
makes the algorithm unbiased and,
in case of functions with dominating mixed smoothness $r$ and compact support in $(0,1)^d$,
further improves the order of its expected error by a logarithmic term.

Section \ref{transformationsection} shows 
that the condition of having support in $(0,1)^d$ can be dropped
by applying a suitable change of variables to the above algorithms.
The resulting algorithms satisfy the error bounds from above
for any function on $[0,1]^d$ with dominating mixed smoothness $r$
or isotropic smoothness $s$. 
Beyond that, the change of variables preserves unbiasedness.

\newpage
\section{The Function Spaces}
\label{functionclassessec}

For natural numbers $r$ and $d$ the Sobolev space $\Hmix$ of dominating
mixed smoothness $r$ is the real vector space
\[
\Hmix = \set{f\in L^2(\IR^d)\mid \diff^\alpha f 
\in L^2(\IR^d) \text{ for every } \alpha \in \set{0,\dots,r}^d}
\]
of $d$-variate, real valued functions, equipped with the scalar product
\[
\mixscalar{f}{g}= \sum\limits_{\alpha \in \set{0,\dots,r}^d} 
\lscalar{\diff^\alpha f}{\diff^\alpha g}
.\]
The scalar product induces the norm
\[
\mixnorm{f}= \left(\sum\limits_{\alpha \in \set{0,\dots,r}^d} 
\lnorm{\diff^\alpha f}^2\right)^{1/2}.
\]

It is known that $\Hmix$ is a Hilbert space and its elements can be considered 
to be continuous functions.
In this thesis, the Fourier transform is the unique continuous linear operator
$\mathcal{F}: L^2(\IR^d)\to L^2(\IR^d)$
satisfying
\[
\mathcal{F}f(y) = \int_{\IR^d}   f(x)\, e^{-2\pi i \scalar{x}{y}}  \, \d x
\]
for integrable $f:\IR^d \to \IR$ and $y\in\IR^d$.
The space $\Hmix$ contains exactly those 
functions $f\in L^2(\IR^d)$ with ${\mathcal{F}f \cdot h_r^{1/2} \in L^2(\IR^d)}$ 
for the Fourier transform $\mathcal{F}f$ of $f$ and the weight function
\[
h_r: \IR^d \to \IR^+,\quad h_r(x)= \sum\limits_{\alpha\in\{0,\dots,r\}^d} 
\prod\limits_{j=1}^{d} |2\pi x_j|^{2\alpha_j} 
= \prod\limits_{j=1}^{d} \sum\limits_{k=0}^{r} |2\pi x_j|^{2k}.
\]
In terms of its Fourier transform, the norm of $f\in\Hmix$ is given by
\[
\mixnorm{f}^2= \int\limits_{\IR^d} \left|\mathcal{F}f(x)\right|^2\cdot h_r(x) \, \d x.
\]

Analogously, the isotropic Sobolev space $\Hiso$ of smoothness $s\in\IN$ is
\[
\Hiso = \set{f\in L^2(\IR^d)\mid \diff^\alpha f \in L^2(\IR^d) 
\text{ for every } \alpha \in \IN_0^d\text{ with } \Vert\alpha\Vert_1\leq s}
,\]
equipped with the scalar product
\[
\isoscalar{f}{g}= \sum\limits_{\Vert\alpha\Vert_1\leq s} 
\lscalar{\diff^\alpha f}{\diff^\alpha g}
\]
and its induced norm $\isonorm{f}$.
For $\alpha\in\IN_0^d$, we will frequently use the abbreviation
$\abs{\alpha}=\Vert\alpha\Vert_1=\sum_{j=1}^{d}\abs{\alpha_j}$.

The space $\Hiso$ is a Hilbert space, too. In the following, we will assume that $s$
is greater than $d/2$.
Then $\Hiso$ also consists of continuous functions, exactly those functions 
$f\in L^2(\IR^d)$ with ${\mathcal{F}f \cdot v_s^{1/2}\in L^2(\IR^d)}$ 
for the Fourier transform $\mathcal{F}f$ of $f$ and the weight function
\[
v_s: \IR^d \to \IR^+,\quad v_s(x)= \sum\limits_{\abs{\alpha}\leq s} 
\prod\limits_{j=1}^{d} |2\pi x_j|^{2\alpha_j}
\asymp \left( 1+\Vert x\Vert_2^2\right)^s
.\]
In terms of its Fourier transform, the norm of $f\in\Hiso$ is given by
\[
\isonorm{f}^2= \int\limits_{\IR^d} \left|\mathcal{F}f(x)\right|^2\cdot v_s(x) \, \d x 
.\]

Furthermore, let $\Cc$ be the real vector space of all continuous real valued functions 
with compact support in $\IR^d$. The spaces $\Hmixunitcomp$ and $\Hisounitcomp$ of 
functions in $\Hmix$ or $\Hiso$ with compact support in the
unit cube are subspaces of $\Cc$.
They can also be considered as subspaces of the Hilbert space
\[
\Hmixunit= \set{f\in L^2([0,1]^d) \mid \diff^\alpha 
f \in L^2([0,1]^d) \text{ for every } \alpha \in \set{0,\dots,r}^d}
,\]
equipped with the scalar product
\[
\mixscalarunit{f}{g}= \sum\limits_{\alpha \in \set{0,\dots,r}^d} 
\lscalarunit{\diff^\alpha f}{\diff^\alpha g}  , 
\]
or the Hilbert space
\[
\Hisounit= \set{f\in L^2([0,1]^d) 
\mid \diff^\alpha f \in L^2([0,1]^d) \text{ for } \alpha 
\in \IN_0^d\text{ with } \abs{\alpha}\leq s}
,\]
with the scalar product
\[
\isoscalarunit{f}{g}= \sum\limits_{\abs{\alpha}\leq s} 
\lscalarunit{\diff^\alpha f}{\diff^\alpha g} 
.\]

\newpage
\section{The Basic Quadrature Rule}
\label{basicquadrulesection}

We introduce a family of deterministic and linear quadrature rules.
This family is fundamental to our studies.
All the algorithms to be presented are based on the following definition.

\medskip 
\noindent
{\bf Algorithm.} \
Let $S\in\IR^{d\times d}$ be invertible and $v$ be a vector in $\IR^d$.
We define
\[
Q_S^v(f)=\frac{1}{|\det S|} 
\sum\limits_{m\in\IZ^d} f\left(S^{-\top}(m+v)\right)
\]
for any admissible input function $f:\IR^d\to\IR$. We call $v$
\emph{shift parameter} and denote by $Q_S$ the algorithm $Q_S^v$ for shift
parameter $v=0$.

\medskip

The matrix $S^{-\top}$ is the transpose of the inverse of $S$.
For now, $S$ can be any invertible matrix. But
later on, it will be a fixed matrix $B$ multiplied with a number
$n^{1/d}$ and a \textit{dilation matrix} $\hat{u}=\diag(u_1,\dots,u_d)$ for
a \textit{dilation parameter} $u\in\IR^d$.
The dilation parameter $u\in\IR^d$ and shift parameter $v\in\IR^d$
are also arbitrary. As we go along, they will be chosen as
independent random variables $U$ and $V$
that are uniformly distributed in $[1,2^{1/d}]^d$ and $[0,1]^d$, respectively.\\

The rule $Q_S^v$ adds up the values of $f$ at the lattice points
$\braces{S^{-\top}\braces{m+v}}$, $m\in\IZ^d$, in the
corner of each parallelepiped $\braces{S^{-\top}\braces{m+v+[0,1]^d}}$
weighted with the volume $\abs{\det S}^{-1}$ of this parallelepiped.
The value $Q_S^v(f)$ hence can be thought of as a Riemann sum of $f$
over $\IR^d$ with respect to the partition
$\set{S^{-\top}\braces{m+v+[0,1]^d}\mid m\in\IZ^d}$.\\

Admissible input functions are, for instance, functions $f$
with compact support.
For such functions the above sum is a finite sum.
To integrate $f$, the algorithm $Q_S^v$ uses the nodes $S^{-\top}(m+v)$, 
where $m\in\IZ^d$ is a lattice point in the 
compact set $\braces{S^\top\left(\supp f\right)-v}$ of volume
$\braces{\det(S)\cdot \lambda^d\left(\supp f\right)}$.
Here, $\lambda^d$ is the Lebesgue measure in $\IR^d$.
The indicated volume is the approximate number of function values computed by $Q_S^v$.
In particular, the number of nodes of $Q_{aS}^v$ for growing $a\geq 1$ is of order $a^d$. 
The following simple lemma gives an exact upper bound, see~\cite{skriganov} for other bounds.

\begin{lemma}
\label{anlemma}
Suppose $f: \IR^d \to \IR$
is supported in an axis-parallel 
cube of edge length $l>0$. For any invertible matrix $S\in\IR^{d\times d}$, 
$v\in\IR^d$ and $a\geq 1$ the quadrature rule $Q_{aS}^v$ uses at most
$
\left(l\cdot\Vert S\Vert_1+1\right)^d\cdot a^d
$
function values of $f$.
\end{lemma}

\begin{proof}
By assumption, $f$ has compact support in $\frac{l}{2}\cdot [-1,1]^d+x_0$ 
for some $x_0\in\IR^d$. The number of computed function values is the
number of points $m\in\IZ^d$ for which $(aS)^{-\top}(m+v)$ is
in $\supp f$ and hence bounded by the size of
\[\begin{split}
M&=\set{m\in\IZ^d \mid (aS)^{-\top}(m+v)\in \frac{l}{2}\cdot [-1,1]^d+x_0}\\
&= \set{m\in\IZ^d \mid m+\left(v-aS^\top x_0\right)\in \frac{al}{2}\cdot S^\top [-1,1]^d}
.\end{split}\]
Since $\Vert S^\top x\Vert_\infty\leq \Vert S^\top\Vert_\infty 
=\Vert S\Vert_1$ for $x\in[-1,1]^d$,
\[
M \subseteq \set{m\in\IZ^d \mid m+\left(v-aS^\top x_0\right)\in 
\left[-\frac{al}{2}\Vert S\Vert_1,\frac{al}{2}\Vert S\Vert_1\right]^d}
\]
and $\vert M\vert \leq \left(al\Vert S\Vert_1+1\right)^d$. 
With $1\leq a$ we get the estimate of Lemma~\ref{anlemma}.
\end{proof}

The error of this algorithm for integration on $\Cc$ can be expressed in terms of the Fourier transform.

\begin{lemma}
\label{errorlemma}
For any invertible matrix $S\in\IR^{d\times d}$, $v\in\IR^d$ and $f\in C_c(\IR^d)$
\[
\left| Q_S^v(f)-I_d(f)\right| 
\leq \sum\limits_{m\in\IZ^d\setminus\set{0}} \left| \mathcal{F}f(Sm)\right|
.\]
\end{lemma}

\begin{proof}
The function $g=f\circ S^{-\top}(\cdot +v)$ is continuous with compact support. 
Hence, the Poisson summation formula and an affine linear substitution $x=S^\top y-v$ yield
\[\begin{split}
Q_S^v(f)&=\frac{1}{\abs{\det S}}\sum\limits_{m\in\IZ^d} g(m) 
= \frac{1}{\abs{\det S}}\sum\limits_{m\in\IZ^d} \mathcal{F}g(m)\\
&= \frac{1}{\abs{\det S}}\sum\limits_{m\in\IZ^d} 
\int\limits_{\IR^d} f\left(S^{-\top}(x+v)\right)\cdot e^{-2\pi i\langle x,m\rangle}
\, \d x \\
&= \sum\limits_{m\in\IZ^d} \int\limits_{\IR^d} f\left(y\right)
\cdot e^{-2\pi i\langle S^\top y-v,m\rangle} \, \d y \\
&= \sum\limits_{m\in\IZ^d} \mathcal{F}f(Sm)\cdot e^{2\pi i\langle v,m\rangle}
,\end{split}\]
if the latter series converges absolutely, see \cite[pp.\,356]{koch}. If not,
the stated inequality is obvious.
This proves the statement, 
since $I_d(f)=\mathcal{F}f(S\cdot 0)\cdot e^{2\pi i\langle v,0\rangle}$.
\end{proof}

\newpage
\section{Frolov's Deterministic Algorithm on $\Hmixunitcomp$}
\label{frolovsrulesection}

It is known how to choose the matrix $S$ in the rule $Q_S^v$ to get a good 
deterministic quadrature rule on $\Hmixunitcomp$. 
Let the matrix $B\in \IR^{d\times d}$ satisfy the following three conditions:
\begin{itemize}
\item[(a)] $B$ is invertible,
\item[(b)] $\left|\prod\limits_{j=1}^{d}(Bm)_j\right|\geq 1$, 
for any $m\in\IZ^d\setminus\set{0}$,
\item[(c)] For any $x,y\in\IR^d$ the box $[x,y]$ with 
volume $V=\prod\limits_{j=1}^{d}|x_j-y_j|$ contains at most $V+1$ 
lattice points $Bm$, $m\in\IZ^d$,
\end{itemize}
where $[x,y]=\set{z\in\IR^d\mid z_j 
\text{ is inbetween of }x_j\text{ and }y_j\text{ for }j=1,\dots,d}$. 
Such a matrix shall be called a \textit{Frolov matrix}. 
Property (b) says that for $n>0$ every point of the lattice $n^{1/d}B\IZ^d$ 
but zero lies in the set $D_n$ of all vectors $x\in\IR^d$ 
with $\prod_{j=1}^{d}\abs{x_j}\geq n$,
the complement of a hyperbolic cross.

\begin{minipage}[h!]{.6\linewidth}
\includegraphics[width=.95\linewidth]{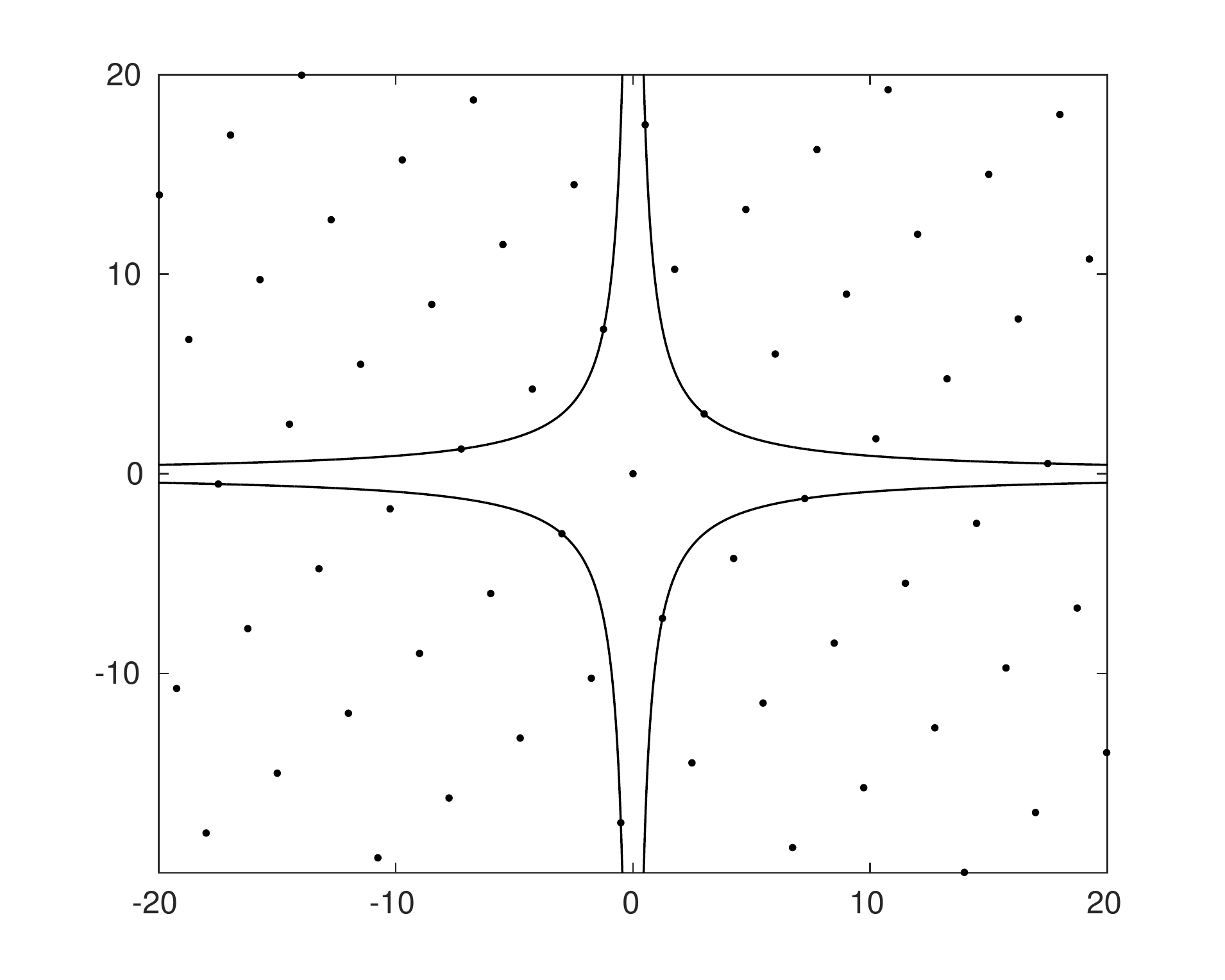}
\end{minipage}
\begin{minipage}[h!]{.35\linewidth}
This graphic shows the lattice $n^{1/d}B\IZ^d$ for $d=2$, $n=9$ and the Frolov matrix
\[B=\begin{pmatrix}
1 & 2-\sqrt{2}\\
1 & 2+\sqrt{2}
\end{pmatrix}
.\]
Except zero, every lattice point lies inside $D_9$.
\end{minipage}

It is known that one can construct such a matrix $B$ in the following way. 
Let $p\in\IZ[x]$ be a polynomial of degree $d$ with leading 
coefficient 1 which is irreducible over $\IQ$ and has $d$ different 
real roots $\zeta_1,\hdots,\zeta_d$. Then the matrix
\[B=\left(\zeta_i^{j-1}\right)_{i,j=1}^d\]
has the desired properties, as shown in \cite[p.\,364]{temlyakovbuch} and \cite{ullrich}. 
In arbitrary dimension $d$ we can choose $p(x)=(x-1)(x-3)\cdot\hdots\cdot(x-2d+1)-1$, 
see \cite{frolov} or \cite{ullrich}, but there are many other possible choices. 
For example, if $d$ is a power of two, we can set $p(x)=2\cos\left(d\cdot\arccos 
(x/2)\right)=2\,T_d(x/2)$, where $T_d$ is the Chebyshev 
polynomial of degree $d$, see \cite[p.\,365]{temlyakovbuch}. 
Then the roots of $p$ are explicitly given by $\zeta_j
=2\cos\left(\frac{2j-1}{2d}\pi\right)$ for $j=1,\hdots,d$.\\

From now on, let $B$ be an arbitrary but fixed, $d$-dimensional Frolov matrix.
Constants may depend on the choice of $B$.

\medskip 
\noindent
{\bf Algorithm.} \
For any natural number $n$, we consider the quadrature rule
$Q_{n^{1/d} B}$ from Section \ref{basicquadrulesection} with shift parameter zero.
This deterministic algorithm is usually referred to as \textit{Frolov's algorithm}.

\medskip

For input functions $f$ with support in $[0,1]^d$ the number of function values computed by $Q_{n^{1/d}B}$
is of order $n$. To be precise, Lemma \ref{anlemma} says that $Q_{n^{1/d}B}$ uses at most
$\left(\Vert B\Vert_1+1\right)^d\cdot n$ function values of $f$.

K.\,K.\,Frolov has already seen in 1976 that the algorithm $Q_{n^{1/d}B}$
is optimal on $\Hmixunitcomp$ in the sense of order of convergence.
It satisfies the following error bound.

\begin{thm}
\label{frolovboundtheorem}
There is some $c>0$ such that for
every $n\geq 2$ and $f\in \Hmixunitcomp$
\[
\abs{Q_{n^{1/d}B}(f)-I_d(f)} \leq\, c \,  n^{-r} 
\, (\log n)^\frac{d-1}{2} \, \mixnormunit{f}
.\]
\end{thm}

See also \cite{frolov} and \cite{ullrich} or my Bachelor thesis for a proof of this error
bound and its optimality.
In fact, this error bound holds uniformly for $Q_{n^{1/d}\hat{u}B}^v$ for any $u\in[1,2^{1/d}]^d$ and $v\in[0,1]^d$,
which is the statement of Theorem~\ref{mixthmworstcase} in Section~\ref{randomdilationsection1}.
Theorem \ref{frolovboundtheorem} is only a special case.\\

But Frolov's algorithm is also optimal among deterministic quadrature rules
on $\Hisounitcomp$ in the sense of order of convergence.
It satisfies:

\begin{thm}
\label{frolovboundisotheorem}
There is some $c>0$ such that for
every $n\geq 2$ and $f\in \Hisounitcomp$
\[
\abs{Q_{n^{1/d}B}(f)-I_d(f)} \leq\, c \,  n^{-s/d} 
\, \isonormunit{f}
.\]
\end{thm}

This is a special case of Theorem~\ref{isothmworstcase} in Section~\ref{randomdilationsection1}.
See \cite{ln} for a proof of the optimality of this order.

\newpage
\section{The Effect of Random Dilations}
\label{randomdilationsection}

We study the impact of random dilations on Frolov's algorithm $Q_{n^{1/d}B}$.

\medskip 
\noindent
{\bf Algorithm.} \
For any natural number $n$ and shift parameter $v\in\IR^d$
we consider the method $Q_{n^{1/d}\hat{U}B}^v$
from Section \ref{basicquadrulesection}
with a dilation parameter $U$ that is uniformly distributed in the box $[1,2^{1/d}]^d$.

\medskip 

For input functions $f$ from $\Hisounitcomp$ or $\Hmixunitcomp$ the
information cost of $Q_{n^{1/d}\hat{U}B}^v$ is roughly between $\det(B) \cdot n$
and $2\cdot\det(B) \cdot n$. More precisely, it uses at most
$2\cdot\left(\Vert B\Vert_1+1\right)^d\cdot n$ function values of $f$.\\

\subsection{Worst Case Errors}
\label{randomdilationsection1}

In the worst case, the error of this method has the same order
of convergence like Frolov's algorithm, both for $\Hmixunitcomp$
and $\Hisounitcomp$.

\begin{thm}
\label{mixthmworstcase}
There is a constant $c>0$ such that for any shift parameter $v\in\IR^d$,
$n\geq 2$ and $f\in \Hmixunitcomp$
\[\begin{split}
\sup\limits_{u\in [1,2^{1/d}]^d} \abs{Q_{n^{1/d}\hat{u}B}^v(f)-I_d(f)} \leq\, c \,  n^{-r} 
\, (\log n)^\frac{d-1}{2} \, \mixnormunit{f}
.\end{split}\]
\end{thm}

\begin{proof}
Let $Q_{n^{1/d}\hat{u}B}^v$ be an arbitrary realization
of the algorithm $Q_{n^{1/d}\hat{U}B}^v$ under consideration.
By Lemma~\ref{errorlemma} and Hölder's inequality,
\[\begin{split}
&\abs{Q_{n^{1/d}\hat{u}B}^v(f)-I_d(f)}^2
\leq \left(\sum\limits_{m\in\IZ^d\setminus\set{0}} \left| 
\mathcal{F}f(n^{1/d}\hat{u}Bm)\right|\right)^2\\
&\leq \left(\sum\limits_{m\in\IZ^d\setminus\set{0}} h_r(n^{1/d}\hat{u}Bm)^{-1}\right)
\cdot\left(\sum\limits_{m\in\IZ^d\setminus\set{0}} h_r(n^{1/d}\hat{u}Bm)\cdot 
\abs{\mathcal{F}f(n^{1/d}\hat{u}Bm)}^2\right)
.\end{split}\]
We first prove that the first factor in this product is bounded above by a constant multiple of $n^{-2r} 
\, (\log n)^{d-1}$.\\
Consider the auxiliary set $N(\beta)=\{x\in\IR^d \mid \lfloor 2^{\beta_j-1}\rfloor \leq|x_j|<2^{\beta_j},1\leq j\leq d\}$
for $\beta\in\IN_0^d$ and $G_n^\beta=\set{m\in\IZ^d\setminus\set{0}\mid n^{1/d}\hat{u}Bm\in N(\beta)}$.
The domain $\IZ^d\setminus\set{0}$ of summation is the disjoint union of all $G_n^\beta$ over $\beta\in\IN_0^d$.\\
For $|\beta|\leq \log_2 n$, the points $x$ in $N(\beta)$ satisfy
$\prod_{j=1}^{d}\abs{x_j}<2^{\abs{\beta}}\leq n$. But the second property of the Frolov matrix $B$ yields
$\prod_{j=1}^d \abs{n^{1/d}u_j(Bm)_j}\geq n$ for any $m\in\IZ^d\setminus\set{0}$.
Hence, $G_n^\beta$ is empty for $|\beta|\leq \log_2 n$.
For $\abs{\beta} >\log_2 n$, any $m\in G_n^\beta$ satisfies
\[
h_r(n^{1/d}\hat{u}Bm)\geq \prod_{j=1}^d \left(1+\lfloor 2^{\beta_j-1}\rfloor^{2r}\right) 
\geq \prod_{j=1}^d 2^{2r(\beta_j-1)} = 2^{2r(|\beta|-d)}
\]
and hence $h_r(n^{1/d}\hat{u}Bm)^{-1}\leq 2^{2r(d-|\beta|)}$.
Because of the third property of the Frolov matrix, we obtain
\[\begin{split}
\abs{G_n^\beta} \leq \abs{\set{m\in\IZ^d\setminus\set{0}\mid \abs{(Bm)_j}
< \frac{2^{\beta_j}}{n^{1/d}}}} \leq 2^{d+\abs{\beta}} n^{-1} +1 \leq 2^{d+1+\abs{\beta}}
n^{-1}
.\end{split}\]
That yields
\[\begin{split}
\sum\limits_{m\in\IZ^d\setminus\set{0}} &h_r\left(n^{1/d}\hat{u}Bm\right)^{-1}\\
&= \sum\limits_{\beta\in\IN_0^d} \sum\limits_{m\in G_n^\beta} h_r(n^{1/d}\hat{u}Bm)^{-1}\\
&= \sum\limits_{|\beta|>\log_2 n}\, \sum\limits_{m\in G_n^\beta} h_r(n^{1/d}\hat{u}Bm)^{-1}\\
&\overset{\star}{\leq} \sum\limits_{|\beta|>\log_2 n}\, 2^{2r(d-|\beta|)}
\cdot n^{-1}\cdot 2^{d+1+|\beta|}\\
&= \sum\limits_{k=\lceil \log_2 n\rceil}^\infty\, 2^{2r(d-k)}
\cdot n^{-1}\cdot 2^{d+1+k}\cdot \abs{\set{\beta\in\IN_0^d\mid |\beta|=k}}\\
&\leq 2^{2rd+d+1} \cdot n^{-1} \sum\limits_{k=\lceil \log_2 n\rceil}^{\infty}
2^{(1-2r)k}\cdot (k+1)^{d-1}\\
&= 2^{2rd+d+1} \cdot n^{-1} 
\sum\limits_{k=0}^{\infty} 2^{(1-2r)(k+\lceil \log_2 n\rceil)}\cdot 
\left(k+1+\lceil \log_2 n\rceil\right)^{d-1}\\
&\leq 2^{2rd+d+1} \cdot n^{-1} \cdot n^{1-2r} 
\cdot \sum\limits_{k=0}^{\infty} 2^{(1-2r)k}\cdot 2^{d-1} 
\cdot (k+1)^{d-1}\cdot\lceil \log_2 n\rceil^{d-1}\\
&\leq 2^{2rd+2d} \cdot n^{-2r} 
\cdot \sum\limits_{k=0}^{\infty} 2^{(1-2r)k} (k+1)^{d-1} 
\left( 2\cdot\frac{\log n}{\log 2}\right)^{d-1}\\
&= \left( 2^{2rd+3d-1}\, (\log 2)^{1-d} 
\sum\limits_{k=0}^{\infty} \left(2^{1-2r}\right)^k (k+1)^{d-1} \right)
\cdot n^{-2r}\, (\log n)^{d-1}
.\end{split}\]
This is the desired estimate, since $2^{1-2r}<1$.\\

We now show that the second factor in the above inequality is bounded above by a constant 
multiple of $\mixnormunit{f}^2$. This proves the theorem.

For $x\in\IR^d$ we have
\[
h_r(x)\cdot \abs{\mathcal{F}f(x)}^{2}
=\sum\limits_{\alpha\in\set{0,\dots,r}^d}\abs{\mathcal{F}D^\alpha f(x)}^2
.\]
The function $g_\alpha=D^\alpha f\circ(n^{1/d}\hat{u}B)^{-\top}$ has compact support
in the parallelepiped $(n^{1/d}\hat{u}B)^\top[0,1]^d$.
Consider the set
$J_n$ of all $k\in\IZ^d$ for which $\braces{k+[0,1]^d}$ has a nonempty intersection 
with $(n^{1/d}\hat{u}B)^\top[0,1]^d$. Then
\[\begin{split}
\abs{\mathcal{F}D^\alpha f \braces{n^{1/d}\hat{u}Bm}}^2
&= \left| \int_{\IR^d} D^\alpha f(y)\cdot 
e^{-2\pi i \langle n^{1/d}\hat{u}Bm,y\rangle} \d y\right|^2\\
&= \left|\frac{1}{\det(n^{1/d}\hat{u}B)}\int_{\IR^d} g_\alpha(x)\cdot 
e^{-2\pi i\langle m,x\rangle} \d x\right|^2\\
&= \left|\frac{1}{\det(n^{1/d}\hat{u}B)}\sum\limits_{k\in J_n} \langle g_\alpha(x),
e^{2\pi i\langle m,\cdot\rangle}\rangle_{L^2\left(k+[0,1]^d\right)}\right|^2\\
&\leq \frac{\abs{J_n}}{\abs{\det(n^{1/d}\hat{u}B)}^2} \sum\limits_{k\in J_n} 
\left|\langle g_\alpha,e^{2\pi i\langle m,\cdot \rangle}\rangle_{L^2\left(k+[0,1]^d\right)}\right|^2
.\end{split}\]
Thus we obtain
\[\begin{split}
\sum\limits_{m\in\IZ^d\setminus\set{0}} &h_r(n^{1/d}\hat{u}Bm)\cdot \abs{\mathcal{F}f(n^{1/d}\hat{u}Bm)}^2
\leq \sum\limits_{m\in\IZ^d} \sum\limits_{\alpha\in\set{0,\dots,r}^d}
\abs{\mathcal{F}D^\alpha f(n^{1/d}\hat{u}Bm)}^2\\
&\leq \frac{\abs{J_n}}{\abs{\det(n^{1/d}\hat{u}B)}^2} \sum\limits_{m\in\IZ^d}
\sum\limits_{\alpha\in\set{0,\dots,r}^d}
\sum\limits_{k\in J_n} \left|\langle g_\alpha,e^{2\pi i\langle m,
\cdot \rangle}\rangle_{L^2\left(k+[0,1]^d\right)}\right|^2\\
&= \frac{\abs{J_n}}{\abs{\det(n^{1/d}\hat{u}B)}^2} \sum\limits_{\alpha\in\set{0,\dots,r}^d} 
\sum\limits_{k\in J_n} \Vert g_\alpha\Vert_{L^2\left(k+[0,1]^d\right)}^2\\
&= \frac{\abs{J_n}}{\abs{\det(n^{1/d}\hat{u}B)}^2} 
\sum\limits_{\alpha\in\set{0,\dots,r}^d} \lnorm{g_\alpha}^2\\
&= \frac{\abs{J_n}}{\abs{\det(n^{1/d}\hat{u}B)}} \sum\limits_{\alpha\in\set{0,\dots,r}^d} 
\lnorm{D^\alpha f}^2\\
&= \frac{\abs{J_n}}{\abs{\det(n^{1/d}\hat{u}B)}} \mixnormunit{f}^2
.\end{split}\]
Since both $\abs{J_n}$ and $\abs{\det(n^{1/d}\hat{u}B)}$ are of order $n$, their ratio is bounded
by a constant and the above inequality yields the statement.
\end{proof}

\begin{thm}
\label{isothmworstcase}
There is a constant $c>0$ such that for any shift parameter $v\in\IR^d$,
$n\in\IN$ and $f\in \Hisounitcomp$
\[\begin{split}
\sup\limits_{u\in [1,2^{1/d}]^d} \abs{Q_{n^{1/d}\hat{u}B}^v(f)-I_d(f)} \leq\, c \,  n^{-s/d}
\, \isonormunit{f}
.\end{split}\]
\end{thm}

\begin{proof}
Let $Q_{n^{1/d}\hat{u}B}^v$ be an arbitrary realization
of the algorithm $Q_{n^{1/d}\hat{U}B}^v$ under consideration.
By Lemma~\ref{errorlemma} and Hölder's inequality,
\[\begin{split}
&\abs{Q_{n^{1/d}\hat{u}B}^v(f)-I_d(f)}^2
\leq \left(\sum\limits_{m\in\IZ^d\setminus\set{0}} 
\left| \mathcal{F}f\braces{n^{1/d}\hat{u}Bm}\right|\right)^2\\
&\leq \left(\sum\limits_{m\in\IZ^d\setminus\set{0}} v_s\braces{n^{1/d}\hat{u}Bm}^{-1}\right)
\cdot\left(\sum\limits_{m\in\IZ^d\setminus\set{0}} v_s\braces{n^{1/d}\hat{u}Bm}
\cdot \abs{\mathcal{F}f\braces{n^{1/d}\hat{u}Bm}}^2\right)
.\end{split}\]
The first factor in this product is bounded above by a constant multiple of 
$n^{-2s/d}$: Since
\[
v_s\braces{n^{1/d}\hat{u}Bm} \geq \Vert n^{1/d}\hat{u}Bm\Vert_2^{2s}
\geq n^{2s/d} \cdot \Vert Bm\Vert_2^{2s}
\geq n^{2s/d}\cdot \Vert B^{-1}\Vert_2^{-2s}\cdot
\Vert m\Vert_2^{2s}
,\]
we have
\[
\sum\limits_{m\in\IZ^d\setminus\set{0}} v_s\braces{n^{1/d}\hat{u}Bm}^{-1}
\leq n^{-2s/d}\cdot \Vert B^{-1}\Vert_2^{2s}\cdot
\sum\limits_{m\in\IZ^d\setminus\set{0}}\Vert m\Vert_2^{-2s}
,\]
where this last series converges for $2s>d$.

We show that the second factor in the above inequality is bounded above 
by a constant multiple of $\isonormunit{f}^2$. This proves the theorem.

For any $x\in\IR^d$ we have
\[
v_s(x)\cdot \abs{\mathcal{F}f(x)}^{2}
=\sum\limits_{\abs{\alpha}\leq s}\abs{\mathcal{F}D^\alpha f(x)}^2
.\]

The function $g_\alpha=D^\alpha f\circ(n^{1/d}\hat{u}B)^{-\top}$ has compact support
in the parallelepiped $(n^{1/d}\hat{u}B)^\top[0,1]^d$.
Again consider the set
$J_n$ of all $k\in\IZ^d$ for which $\braces{k+[0,1]^d}$ has a nonempty intersection 
with $(n^{1/d}\hat{u}B)^\top[0,1]^d$.\\
We have the estimate
\[\begin{split}
\abs{\mathcal{F}D^\alpha f\braces{n^{1/d}\hat{u}Bm}}^2
&= \left| \int_{\IR^d} D^\alpha f(y)\cdot
e^{-2\pi i \langle n^{1/d}\hat{u}Bm,y\rangle} \d y\right|^2\\
&= \left|\frac{1}{\det(n^{1/d}\hat{u}B)}\int_{\IR^d} g_\alpha(x)\cdot 
e^{-2\pi i\langle m,x\rangle} \d x\right|^2\\
&= \left|\frac{1}{\det(n^{1/d}\hat{u}B)}\sum\limits_{k\in J_n} 
\langle g_\alpha(x),e^{2\pi i\langle m,\cdot \rangle}\rangle_{L^2\left(k+[0,1]^d\right)}\right|^2\\
&\leq \frac{\abs{J_n}}{\abs{\det(n^{1/d}\hat{u}B)}^2} \sum\limits_{k\in J_n} 
\left|\langle g_\alpha,e^{2\pi i\langle m,\cdot \rangle}\rangle_{L^2\left(k+[0,1]^d\right)}\right|^2
.\end{split}\]
Thus we obtain
\[\begin{split}
&\sum\limits_{m\in\IZ^d\setminus\set{0}} v_s\braces{n^{1/d}\hat{u}Bm}\cdot \abs{\mathcal{F}f\braces{n^{1/d}\hat{u}Bm}}^2
\leq \sum\limits_{m\in\IZ^d} \sum\limits_{\abs{\alpha}\leq s}
\abs{\mathcal{F}D^\alpha f\braces{n^{1/d}\hat{u}Bm}}^2\\
&\leq \frac{\abs{J_n}}{\abs{\det(n^{1/d}\hat{u}B)}^2} \sum\limits_{m\in\IZ^d}
\sum\limits_{\abs{\alpha}\leq s}\,
\sum\limits_{k\in J_n} \left|\langle g_\alpha,e^{2\pi i\langle m,
\cdot \rangle}\rangle_{L^2\left(k+[0,1]^d\right)}\right|^2\\
&= \frac{\abs{J_n}}{\abs{\det(n^{1/d}\hat{u}B)}^2} \sum\limits_{\abs{\alpha}
\leq s}\,
\sum\limits_{k\in J_n} \Vert g_\alpha\Vert_{L^2\left(k+[0,1]^d\right)}^2
= \frac{\abs{J_n}}{\abs{\det(n^{1/d}\hat{u}B)}^2} \sum\limits_{\abs{\alpha}
\leq s} 
\lnorm{g_\alpha}^2\\
&= \frac{\abs{J_n}}{\abs{\det(n^{1/d}\hat{u}B)}} \sum\limits_{\abs{\alpha}
\leq s} 
\lnorm{D^\alpha f}^2
= \frac{\abs{J_n}}{\abs{\det(n^{1/d}\hat{u}B)}} \isonormunit{f}^2
.\end{split}\]
Since both $\abs{J_n}$ and $\abs{\det(n^{1/d}\hat{u}B)}$ are of order $n$, their ratio is bounded
by a constant and the above inequality yields the statement.
\end{proof}

\subsection{Expected Errors}
\label{randomdilationsection2}

In expectation, the random dilations improve the order of the error
of Frolov's algorithm by $1/2$ for
both $\Hmixunitcomp$ and $\Hisounitcomp$.
These results are based on the following general error bound for
continuous functions with compact support. Recall that
$D_n$ is the set of all $x\in\IR^d$ 
with $\prod_{j=1}^{d}\abs{x_j}\geq n$.

\begin{thm}  
\label{keyprop}
There is a constant $c>0$ such that for every $n\in\IN$, shift parameter $v\in\IR^d$ and $f\in\Cc$
\[\IE \abs{Q_{n^{1/d}\hat{U}B}^v(f)-I_d(f)} \leq\, c \,  n^{-1} \cdot \int_{D_n} \abs{\mathcal{F}f(x)}\, \d x
.\]
\end{thm} 

\begin{proof}
Thanks to Lemma~\ref{errorlemma} 
and the monotone convergence theorem we have
\[\begin{split}
\IE \abs{Q_{n^{1/d}\hat{U}B}^v(f)-I_d(f)}
&\leq \IE 
\left(\sum\limits_{m\in\IZ^d\setminus\set{0}} \abs{\mathcal{F}f\braces{n^{1/d}\hat{U}Bm}} \right)\\
&= \sum\limits_{m\in\IZ^d\setminus\set{0}} \IE \abs{\mathcal{F}f\braces{n^{1/d}\hat{U}Bm}}
.\end{split}\]
Since each $n^{1/d}\hat{U}Bm$ is uniformly distributed in the 
box $[n^{1/d}Bm,(2n)^{1/d}Bm]$ with volume 
$\left(2^{1/d}-1\right)^d\cdot\abs{\prod_{j=1}^dn^{1/d} (Bm)_j}$, this series equals
\[\begin{split}
&\frac{1}{\left(2^{1/d}-1\right)^d} \sum\limits_{m\in\IZ^d\setminus\set{0}}\, 
\int\limits_{[n^{1/d}Bm,(2n)^{1/d}Bm]} \frac{\abs{\mathcal{F}f(x)}}{\prod_{j=1}^d
\abs{n^{1/d}(Bm)_j}} \, \d x \\
&\leq \frac{1}{\left(2^{1/d}-1\right)^d} 
\sum\limits_{m\in\IZ^d\setminus\set{0}}\, \int\limits_{[n^{1/d}Bm,(2n)^{1/d}Bm]} 
\frac{\abs{\mathcal{F}f(x)}}{\prod_{j=1}^d 2^{-1/d}\abs{x_j}} \, \d x\\
&= \frac{2}{\left(2^{1/d}-1\right)^d}\cdot \int_{\IR^d} 
\frac{\abs{\mathcal{F}f(x)}}{\prod_{j=1}^d \abs{x_j}}\cdot 
\abs{\set{m\in\IZ^d\setminus\set{0}\mid x\in [n^{1/d}Bm,(2n)^{1/d}Bm]}} \, \d x
\end{split}\]
\[\begin{split}
&= \frac{2}{\left(2^{1/d}-1\right)^d}\cdot \int_{\IR^d} 
\frac{\abs{\mathcal{F}f(x)}}{\prod_{j=1}^d \abs{x_j}}\cdot 
\abs{\set{m\in\IZ^d\setminus\set{0}\mid Bm\in 
\left[\frac{x}{(2n)^{1/d}},\frac{x}{n^{1/d}}\right]}}\, \d x  
.\end{split}\]
Thanks to the properties of the Frolov matrix $B$, if $\prod_{j=1}^d \abs{x_j}<n$, 
the latter set is empty and otherwise contains no more 
than $\prod_{j=1}^d \abs{\frac{x_j}{n^{1/d}}}+1\leq 2 n^{-1} \prod_{j=1}^d \abs{x_j}$ points.
Thus, we arrive at
\[
\IE \abs{Q_{n^{1/d}\hat{U}B}^v(f)-I_d(f)} \leq
\frac{4}{\left(2^{1/d}-1\right)^d}\cdot n^{-1} \int_{D_n} \abs{\mathcal{F}f(x)} \, \d x 
\]
and the theorem is proven.
\end{proof}

Additional differentiability properties of the function $f\in\Cc$ result in decay 
properties of its Fourier transform $\mathcal{F}f$. This leads to estimates of the 
integral $\int_{D_n} \abs{\mathcal{F}f(x)} \, \d x$.
Hence, the general upper bound for the error of $Q_{n^{1/d}\hat{U}B}^v(f)$ in Theorem
\ref{keyprop} adjusts to the differentiability of $f$.
Two such examples are functions from $\Hmixcomp$ and $\Hisocomp$.

\begin{lemma}
\label{intlemmamix}
There is some $c>0$ such that
for each $n\geq 2$ and $f\in\Hmixunitcomp$
\[
\int_{D_n} \abs{\mathcal{F}f(x)} \, \d x  \leq c \,  n^{-r+1/2}
\, \left(\log n\right)^{\frac{d-1}{2}} \, \mixnormunit{f}
.\]
\end{lemma}

\begin{proof}
Applying Hölder's inequality and a linear substitution $x=n^{1/d}By$ 
to the above integral, we get
\[\begin{split}
&\left( \int_{D_n} \abs{\mathcal{F}f(x)} \, \d x  \right)^2
= \left( \int_{D_n} h_r(x)^{-1/2} \cdot
\abs{\mathcal{F}f(x)} h_r(x)^{1/2}\, \d x  \right)^2\\
&\leq \braces{\int_{D_n} h_r(x)^{-1} \, \d x} \mixnorm{f}^2
=  n \abs{\det B} \left(\int_{G} h_r(n^{1/d}By)^{-1} \, \d y \right) \mixnorm{f}^2
\end{split}\]
with $G=B^{-1}D_1$ being the set of all $y\in\IR^d$ with 
$\prod_{j=1}^{d}\abs{(By)_j}\geq 1$. It it thus sufficient to prove that the 
integral $\int_{G} h_r(n^{1/d}By)^{-1} \, \d y$ is bounded by a constant 
multiple of $n^{-2r}\, \left(\log n\right)^{d-1}$.\\
Again consider the auxiliary set $N(\beta)=\{x\in\IR^d \mid \left[2^{\beta_j-1}\right]
\leq|x_j|<2^{\beta_j},1\leq j\leq d\}$ for $\beta\in\IN_0^d$
and the set $G_n^\beta=\set{y\in G\mid n^{1/d}By\in N(\beta)}$.
Similar to the proof of Theorem \ref{mixthmworstcase},
the domain $G$ of integration is the disjoint union of all $G_n^\beta$ over $\beta\in\IN_0^d$,
where $G_n^\beta=\emptyset$, if $\abs{\beta}\leq \log_2 n$, and otherwise
the integrand is bounded above by $2^{2r(d-\abs{\beta})}$ for $y\in G_n^\beta$.
On the other hand,
\[\begin{split}
&\lambda^d(G_n^\beta)\leq \lambda^d\left((n^{1/d}B)^{-1}N(\beta)\right) 
= n^{-1}\cdot |\det B|^{-1}\cdot \lambda^d(N(\beta)) \\
&= n^{-1}\cdot |\det B|^{-1}\cdot 2^d \cdot \prod_{j=1}^d\left(2^{\beta_j}-
\left[2^{\beta_j-1}\right]\right)
\leq n^{-1}\cdot |\det B|^{-1}\cdot 2^d\cdot 2^{|\beta|}
.\end{split}\]
Like in the proof of Theorem \ref{mixthmworstcase}, we obtain
\[\begin{split}
&\int\limits_{G} h_r(n^{1/d}By)^{-1} \, \d y 
= \sum\limits_{\beta\in\IN_0^d}\, \int_{G_n^\beta} h_r(n^{1/d}By)^{-1} \, \d y \\
&= \sum\limits_{|\beta|>\log_2 n}\ \int_{G_n^\beta} h_r(n^{1/d}By)^{-1} \, \d y \\
&\leq \sum\limits_{|\beta|>\log_2 n} 2^{2r(d-|\beta|)}
\cdot n^{-1}\cdot |\det B|^{-1}\cdot 2^{d+\beta}\\
&= |\det B|^{-1}\cdot 2^{-1}\cdot \sum\limits_{|\beta|>\log_2 n} 2^{2r(d-|\beta|)}
\cdot n^{-1}\cdot 2^{d+1+\abs{\beta}}\\
&\overset{\star}{\leq}
\left( 2^{2rd+3d-2} |\det B|^{-1} (\log 2)^{1-d} 
\sum\limits_{k=0}^{\infty} \left(2^{1-2r}\right)^k (k+1)^{d-1} \right)
\, n^{-2r}\, (\log n)^{d-1}
,\end{split}\]
where the constant is finite, since $2^{1-2r}<1$.
\end{proof}

Combining Theorem~\ref{keyprop} and Lemma~\ref{intlemmamix} yields:

\begin{thm}
\label{mixthm}
There is a constant $c>0$ such that 
for every $n\geq 2$, shift parameter $v\in\IR^d$ and $f\in \Hmixunitcomp$
\[
\IE \abs{Q_{n^{1/d}\hat{U}B}^v(f)-I_d(f)} \leq\, c \, n^{-r-1/2} 
\, (\log n)^\frac{d-1}{2} \, \mixnormunit{f}
.\]
\end{thm}

If, however, the integrand is from the space $\Hisounitcomp\subseteq\Cc$, 
the following lemma holds.

\begin{lemma}
\label{intlemmaiso}
There is some $c>0$
such that for each $n\in\IN$ and $f\in\Hisounitcomp$
\[
\int_{D_n} \abs{\mathcal{F}f(x)} \, \d x  \leq c \,  n^{-s/d+1/2} \, \isonormunit{f}
.\]
\end{lemma}

\begin{proof}
Like in Lemma~\ref{intlemmamix}, we apply Hölder's inequality and get
\[\begin{split}
&\left( \int_{D_n} \abs{\mathcal{F}f(x)} \, \d x  \right)^2
=\left( \int_{D_n} \abs{\mathcal{F}f(x)} v_s(x)^{1/2} 
\cdot v_s(x)^{-1/2} \, \d x  \right)^2\\
&\leq \left(\int_{D_n} v_s(x)^{-1} \, \d x \right)\cdot \isonorm{f}^2
\leq \tilde{c}\cdot \left(\int_{D_n} \left(1+\Vert x\Vert_2^2
\right)^{-s} \, \d x \right)\cdot \isonormunit{f}^2
,\end{split}\]
for some $\tilde{c}>0$. 
Since $\Vert x\Vert_2\geq \max\set{\abs{x_j}\mid j=1,\dots,n}\geq n^{1/d}$ for $x\in D_n$, 
the set $D_n$ is a subset of $\set{x\in\IR^d:\,\Vert x\Vert_2\geq n^{1/d}}$
and the latter integral in the above integral is less than
\[\begin{split}
&\int\limits_{\set{x\in\IR^d:\,\Vert x\Vert_2\geq n^{1/d}}} 
\left(1+\Vert x\Vert_2^2\right)^{-s} \, \d x
= \int\limits_{n^{1/d}}^{\infty}\int\limits_{S_{d-1}} 
\braces{1+R^2}^{-s}\cdot R^{d-1} \, \d \sigma \, \d R \\
&= \sigma\left(S_{d-1}\right) \int\limits_{n^{1/d}}^{\infty} 
\braces{1+R^2}^{-s}\cdot R^{d-1} \, \d R 
\leq \sigma\left(S_{d-1}\right) \int\limits_{n^{1/d}}^{\infty} R^{-2s+d-1} \, \d R 
\leq \hat{c} \cdot n^{-2s/d+1}
,\end{split}\]
for some $\hat{c}>0$, since $-2s+d-1<-1$.
\end{proof}

In this case, combining Theorem~\ref{keyprop} and Lemma~\ref{intlemmaiso} yields:

\begin{thm}
\label{isothm}
There is a constant $c>0$ such that for 
every $n\in\IN$, shift parameter $v\in\IR^d$ and $f\in \Hisounitcomp$
\[\begin{split}
\IE \abs{Q_{n^{1/d}\hat{U}B}^v(f)-I_d(f)} \leq\, c \,  n^{-s/d-1/2} \, \isonormunit{f}
.\end{split}\]
\end{thm}

We remark that the Frolov properties of the matrix $B$ are not needed to get this
estimate on $\Hisounitcomp$, although they are essential for the upper bound
on $\Hmixunitcomp$ from Theorem~\ref{mixthm}. As seen in the proof of Lemma~\ref{intlemmaiso}
and contrarily to Lemma~\ref{intlemmamix},
we do not need that the lattice points of $n^{1/d}B\IZ^d\setminus\set{0}$
lie in $D_n$, but only that they lie in the 
bigger set $\set{x\in\IR^d\mid \Vert x\Vert_2\geq n^{1/d}}$. For example, the 
identity matrix would do. But if $B$ is a Frolov matrix, $Q_{n^{1/d}\hat{u}B}^v$ works 
universally for $\Hmixcomp$ and $\Hisocomp$. Furthermore, the Frolov properties
of $B$ prevent extremely large jumps of the number of nodes of
$Q_{n^{1/d}\hat{u}B}^v=Q_{n^{1/d}\hat{u}B}^v$ for small changes of 
the dilation parameter $u\in[1,2^{1/d}]^d$.

\newpage
\section{Further Improvements through Random Shifts}
\label{randomshiftsection}

Now we also choose the
shift parameter $v$ in $Q_{n^{1/d}\hat{U}B}^v$ at random.

\medskip 
\noindent
{\bf Algorithm.} \
For any  natural number $n$
we consider the method $Q_{n^{1/d}\hat{U}B}^V$
from Section \ref{basicquadrulesection}
with independent dilation parameter $U$, uniformly distributed in $[1,2^{1/d}]^d$,
and shift parameter $V$, uniformly distributed in $[0,1]^d$.
\medskip 

For input functions $f$ from $\Hisounitcomp$ or $\Hmixunitcomp$ the
information cost of the method $Q_{n^{1/d}\hat{U}B}^V$ is again of order $n$.\\

The first advantage of this method is its unbiasedness.

\begin{prop}
\label{Munbiased}
Let $S\in\IR^{d\times d}$ be an
invertible matrix. For any $f\in L^1\braces{\IR^d}$, the method $Q_S^V$ satisfies
\[
\IE \braces{Q_S^V(f)}=I_d(f)
.\]
In particular, the method $Q_{n^{1/d}\hat{U}B}^V$ is well-defined and unbiased on $L^1\braces{\IR^d}$.
\end{prop}

\begin{proof}
By the monotone convergence theorem,
\[\begin{split}
&\IE \braces{\sum\limits_{m\in\IZ^d} \frac{1}{\abs{\det S}} 
\abs{f\left( S^{-\top}(m+V) \right)}}
=\sum\limits_{m\in\IZ^d} \IE \braces{\frac{1}{\abs{\det S}} 
\abs{f\left( S^{-\top}(m+V) \right)}} \\
&= \sum\limits_{m\in\IZ^d} \frac{1}{\abs{\det S}} 
\int_{[0,1]^d} \abs{f\left( S^{-\top}(m+x) \right)} \, \d x \\
&= \sum\limits_{m\in\IZ^d}\ \int_{S^{-\top}\braces{m+[0,1]^d}} 
\abs{f(y)} \, \d y
= \int_{\IR^d} \abs{f(y)} \, \d y\,
< \infty
.\end{split}\]
The series $Q_S^V(f)= \sum\limits_{m\in\IZ^d} \frac{1}{\abs{\det S}} 
f\left( S^{-\top}(m+V) \right)$ hence converges absolutely almost surely
and is dominated by the integrable function
$\sum\limits_{m\in\IZ^d} \frac{1}{\abs{\det S}} 
\abs{f\left( S^{-\top}(m+V) \right)}$.
We can thus apply Lebesgue's dominated convergence theorem to get
\[\begin{split}
\IE \braces{Q_{S}^V(f)}
&= \sum\limits_{m\in\IZ^d} \frac{1}{\abs{\det S}} 
\int_{[0,1]^d} f\left( S^{-\top}(m+x) \right) \, \d x\\
&= \sum\limits_{m\in\IZ^d}\ \int_{S^{-\top}\braces{m+[0,1]^d}} 
f(y) \, \d y
= \int_{\IR^d} f(y) \, \d y
= I_d(f)
,\end{split}\]
for the expected value of the general algorithm at $f$.\\
For the method $Q_{n^{1/d}\hat{U}B}^V$, Fubini's theorem yields
\[\begin{split}
\IE \left(Q_{n^{1/d}\hat{U}B}^V(f)\right)
= \IE_U \IE_V \left(Q_{n^{1/d}\hat{U}B}^V(f)\right)
= \IE_U \braces{I_d(f)}
= I_d(f)
,\end{split} \]
as claimed. In particular, $Q_{n^{1/d}\hat{U}B}^V(f)$ is almost surely absolutely convergent.
\end{proof}

The worst case error of this method, too, has
the order $n^{-r}\braces{\log n}^{(d-1)/2}$ on $\Hmixunitcomp$
and $n^{-s/d}$ on $\Hisounitcomp$. This is a direct consequence
of Theorem~\ref{mixthmworstcase} and Theorem~\ref{isothmworstcase}.

\begin{cor}
\label{mixthmworstcase2}
There is some $c>0$ such that for any
$n\geq 2$ and $f\in \Hmixunitcomp$
\[\begin{split}
\sup\limits_{(u,v)\in[1,2^{1/d}]^d\times[0,1]^d} \abs{Q_{n^{1/d}\hat{u}B}^v(f)-I_d(f)} \leq\, c \,  n^{-r} 
\, (\log n)^\frac{d-1}{2} \, \mixnormunit{f}
.\end{split}\]
\end{cor}

\begin{cor}
\label{isothmworstcase2}
There is some $c>0$ such that for any
$n\in\IN$ and $f\in \Hisounitcomp$
\[\begin{split}
\sup\limits_{(u,v)\in[1,2^{1/d}]^d\times[0,1]^d} \abs{Q_{n^{1/d}\hat{u}B}^v(f)-I_d(f)} \leq\, c \,  n^{-s/d}
\, \isonormunit{f}
.\end{split}\]
\end{cor}

The second advantage of this method is the slightly better
convergence order of its expected error on $\Hmixunitcomp$.
As proven by Mario Ullrich in \cite{ullrichneu},
the expected error $\IE \abs{Q_{n^{1/d}\hat{U}B}^V(f)-I_d(f)}$
of $Q_{n^{1/d}\hat{U}B}^V$ in $f\in\Hmixunitcomp$ is bounded
above by a constant multiple of $n^{-r-1/2} \, \mixnormunit{f}$
instead of a constant multiple of
$n^{-r-1/2}\,(\log n)^\frac{d-1}{2} \, \mixnormunit{f}$.

The proof even shows that the quantity $\braces{\IE \abs{Q_{n^{1/d}\hat{U}B}^V(f)-I_d(f)}^2}^{1/2}$
satisfies this bound.
This is a stronger statement, as implied by Hölder's inequality.\\

In Lemma~\ref{errorlemma}, the absolute error of $Q_S^v$ for integration
on $\Cc$ was expressed in terms of the Fourier transform.
The same can be done for
the expected quadratic error of $Q_S^V$.

\begin{lemma}
\label{errorlemma2}
For any invertible matrix
$S\in\IR^{d\times d}$ and $f\in \Cc$ we have
\[
\IE \abs{Q_S^V(f)-I_d(f)}^2 
= \sum\limits_{m\in\IZ^d\setminus\set{0}} \abs{\mathcal{F}f(Sm)}^2
.\]
\end{lemma}

\begin{proof}
Since the expected value of $Q_S^V(f)$ is $I_d(f)$, we have
\[
\IE \abs{Q_S^V(f)-I_d(f)}^2
= \Var \braces{Q_S^V(f)}
= \IE \braces{Q_S^V(f)^2} - I_d(f)^2
.\]
The algorithm $Q_S^v(f)$ considered as a function of $v\in[0,1]^d$
is a finite sum of functions
$|\det S|^{-1}\,f\left(S^{-\top}(k+\cdot)\right)$ in $L^2\braces{[0,1]^d}$
and hence itself in $L^2\braces{[0,1]^d}$.
Parseval's identity states
\[\begin{split}
\IE \braces{Q_S^V(f)^2}
= \lnormunit{Q_S^{\braces{\cdot}}(f)}^2
= \sum\limits_{m\in\IZ^d}\abs{\lscalarunit{Q_S^{\braces{\cdot}}(f)}{e^{2\pi i \scalar{m}{\cdot}}}}^2
.\end{split}\]
For each index $m\in\IZ^d$ we have the equality
\[\begin{split}
\lscalarunit{Q_S^{\braces{\cdot}}(f)}{e^{2\pi i \scalar{m}{\cdot}}}
&= \abs{\det S}^{-1} \sum\limits_{k\in\IZ^d}
\int_{[0,1]^d} f\left(S^{-\top}(k+v)\right)\, e^{-2\pi i \scalar{m}{v}}\d v\\
&= \abs{\det S}^{-1}
\int_{\IR^d} f\left(S^{-\top}v\right)\, e^{-2\pi i \scalar{m}{v}}\d v\\
&= \int_{\IR^d} f\left(v\right)\, e^{-2\pi i \scalar{Sm}{v}}\d v\\
&= \mathcal{F}f(Sm)
.\end{split}\]
We arrive at
\[
\IE \abs{Q_S^V(f)-I_d(f)}^2
= \sum\limits_{m\in\IZ^d} \abs{\mathcal{F}f(Sm)}^2 - I_d(f)^2
= \sum\limits_{m\in\IZ^d\setminus\set{0}} \abs{\mathcal{F}f(Sm)}^2
,\]
which is what had to be proven.
\end{proof}

Now follows an analogue of Theorem~\ref{keyprop}
for expected quadratic errors.
Like before, the general error bound for
continuous functions with compact support
adjusts to additional smoothness
properties.

\begin{thm}  
\label{keyprop2}
There is some $c>0$ such that for every $n\in\IN$ and $f\in\Cc$
\[
\IE \abs{Q_{n^{1/d}\hat{U}B}^V(f)-I_d(f)}^2
\leq\, c \,  n^{-1} \, \Vert\mathcal{F}f\Vert_{L^2\braces{D_n}}^2
.\]
\end{thm}

\begin{proof}
By Lemma~\ref{errorlemma2} 
and the monotone convergence theorem,
\[\begin{split}
\IE \abs{Q_{n^{1/d}\hat{U}B}^V(f)-I_d(f)}^2
&= \IE_U \IE_V \abs{Q_{n^{1/d}\hat{U}B}^V(f)-I_d(f)}^2\\
&= \IE_U \left(\sum\limits_{m\in\IZ^d\setminus\set{0}} \abs{\mathcal{F}f(n^{1/d}\hat{U}Bm)}^2 \right)\\
&= \sum\limits_{m\in\IZ^d\setminus\set{0}} \IE_U \abs{\mathcal{F}f(n^{1/d}\hat{U}Bm)}^2
.\end{split}\]
Since each $n^{1/d}\hat{U}Bm$ is uniformly distributed in the 
box $[n^{1/d}Bm,(2n)^{1/d}Bm]$ with volume 
$\left(2^{1/d}-1\right)^d\cdot\abs{\prod_{j=1}^d n^{1/d} (Bm)_j}$, this series equals
\[\begin{split}
&\frac{1}{\left(2^{1/d}-1\right)^d} \sum\limits_{m\in\IZ^d\setminus\set{0}}\, 
\int\limits_{[n^{1/d}Bm,(2n)^{1/d}Bm]} \frac{\abs{\mathcal{F}f(x)}^2}{\prod_{j=1}^d
\abs{n^{1/d}(Bm)_j}} \, \d x \\
&\leq \frac{1}{\left(2^{1/d}-1\right)^d} 
\sum\limits_{m\in\IZ^d\setminus\set{0}}\, \int\limits_{[n^{1/d}Bm,(2n)^{1/d}Bm]} 
\frac{\abs{\mathcal{F}f(x)}^2}{\prod_{j=1}^d 2^{-1/d}\abs{x_j}} \, \d x\\
&= \frac{2}{\left(2^{1/d}-1\right)^d}\cdot \int_{\IR^d} 
\frac{\abs{\mathcal{F}f(x)}^2}{\prod_{j=1}^d \abs{x_j}}\cdot 
\abs{\set{m\in\IZ^d\setminus\set{0}\mid x\in [n^{1/d}Bm,(2n)^{1/d}Bm]}} \, \d x\\
&= \frac{2}{\left(2^{1/d}-1\right)^d}\cdot \int_{\IR^d} 
\frac{\abs{\mathcal{F}f(x)}^2}{\prod_{j=1}^d \abs{x_j}}\cdot 
\abs{\set{m\in\IZ^d\setminus\set{0}\mid Bm\in
\left[\frac{x}{(2n)^{1/d}},\frac{x}{n^{1/d}}\right]}}\, \d x  
.\end{split}\]
Thanks to the properties of the Frolov matrix $B$, if $\prod_{j=1}^d \abs{x_j}<n$, 
the latter set is empty and otherwise contains no more 
than $\prod_{j=1}^d \abs{\frac{x_j}{n^{1/d}}}+1\leq 2 n^{-1} \prod_{j=1}^d \abs{x_j}$ points.
Thus, we arrive at
\[
\IE \abs{Q_{n^{1/d}\hat{U}B}^V(f)-I_d(f)} \leq
\frac{4}{\left(2^{1/d}-1\right)^d}\cdot n^{-1} \int_{D_n} \abs{\mathcal{F}f(x)}^2 \, \d x 
\]
and the theorem is proven.
\end{proof}

Finally, we can prove the stated upper bound 
for the expected quadratic error of the
method $Q_{n^{1/d}\hat{U}B}^V$.

\begin{thm}
\label{mixthm2}
There is some $c>0$ such that 
for every $n\geq 2$ and $f\in \Hmixunitcomp$
\[
\braces{\IE \abs{Q_{n^{1/d}\hat{U}B}^V(f)-I_d(f)}^2}^{1/2} \leq\, c \, n^{-r-1/2} 
\, \mixnormunit{f}
.\]
\end{thm}

\begin{proof}
If $c_0$ is the constant of Theorem~\ref{keyprop2}, we have the upper bound
\[\begin{split}
\IE \abs{Q_{n^{1/d}\hat{U}B}^V(f)-I_d(f)}^2
&\leq c_0 \,  n^{-1} \, \Vert\mathcal{F}f\Vert_{L^2\braces{D_n}}^2\\
&= c_0\, n^{-1} \, \int_{D_n} h_r(x)^{-1}\cdot \left|\mathcal{F}f(x)\right|^2\, h_r(x)\, \d x\\
&\leq\, c_0 \,  n^{-1} \, \Vert h_r^{-1}\Vert_{L^\infty\braces{D_n}} \cdot \int_{\IR^d}\left|\mathcal{F}f(x)\right|^2\, h_r(x)\, \d x
\end{split}\]
for the expected quadratic error.\\
Since $h_r(x)\geq n^{2r}$ for $x\in D_n$, we obtain the estimate
\[
\braces{\IE \abs{Q_{n^{1/d}\hat{U}B}^V(f)-I_d(f)}^2}^{1/2} \leq c_0^{1/2}\, n^{-r-1/2} 
\, \mixnormunit{f}
.\]
which proves the theorem.
\end{proof}

The method $Q_{n^{1/d}\hat{U}B}^V$ is also optimal for $\Hisounitcomp$.
This can be derived from Theorem~\ref{keyprop2} using the same short argument
from the proof of Theorem~\ref{mixthm2}.
The upper bound for $\IE \abs{Q_{n^{1/d}\hat{U}B}^V(f)-I_d(f)}$
is also a direct consequence of Theorem~\ref{isothm}.

\begin{thm}
\label{isothm2}
There is some $c>0$ such that 
for every $n\in\IN$ and $f\in \Hisounitcomp$
\[
\braces{\IE \abs{Q_{n^{1/d}\hat{U}B}^V(f)-I_d(f)}^2}^{1/2} \leq\, c \, n^{-s/d-1/2} 
\, \isonormunit{f}
.\]
\end{thm}

See \cite{ln} for a proof of the optimality of this order.

\newpage
\section{Transformations to $\Hmixunit$}
\label{transformationsection}

We can transform the above methods $Q_{n^{1/d}B}$ and $Q_{n^{1/d}\hat{U}B}^V$
such that their errors satisfy the same upper bounds for the full spaces
$\Hmixunit$ and $\Hisounit$, that the original algorithms satisfy
for the subspaces $\Hmixunitcomp$ and $\Hisounitcomp$.
This is done by a standard method, which was already used in 
\cite[pp.\,359]{temlyakov} to transform Frolov's deterministic algorithm.
It preserves the unbiasedness of the algorithm $Q_{n^{1/d}\hat{U}B}^V$.

To that end let $\psi:\IR\to\IR$ be an infinitely differentiable function 
such that $\psi|_{(-\infty,0)}=0$, $\psi|_{(1,\infty)}=1$ 
and $\psi|_{(0,1)}:(0,1)\to(0,1)$ is a diffeomorphism. For example, we can choose
\[
h(x)=\begin{cases}
e^\frac{1}{(2x-1)^2-1} & \text{if } x\in(0,1),\\
0 & \text{else,}
\end{cases}
\quad\quad
\psi(x)=\frac{\int_{-\infty}^x h(t) \, \d t}{\int_{-\infty}^{\infty} h(t) \, \d t}
\]
for $x\in\IR$. Like $h$ also $\psi$ is infinitely differentiable and apparently 
satisfies $\psi|_{(-\infty,0)}=0$ and $\psi|_{(1,\infty)}=1$. 
Since the derivative of $\psi$ is strictly positive on $(0,1)$, 
it is strictly increasing and a bijection of $(0,1)$ and its inverse function is smooth.

\begin{minipage}[h!]{.49\linewidth}
\includegraphics[width=\linewidth]{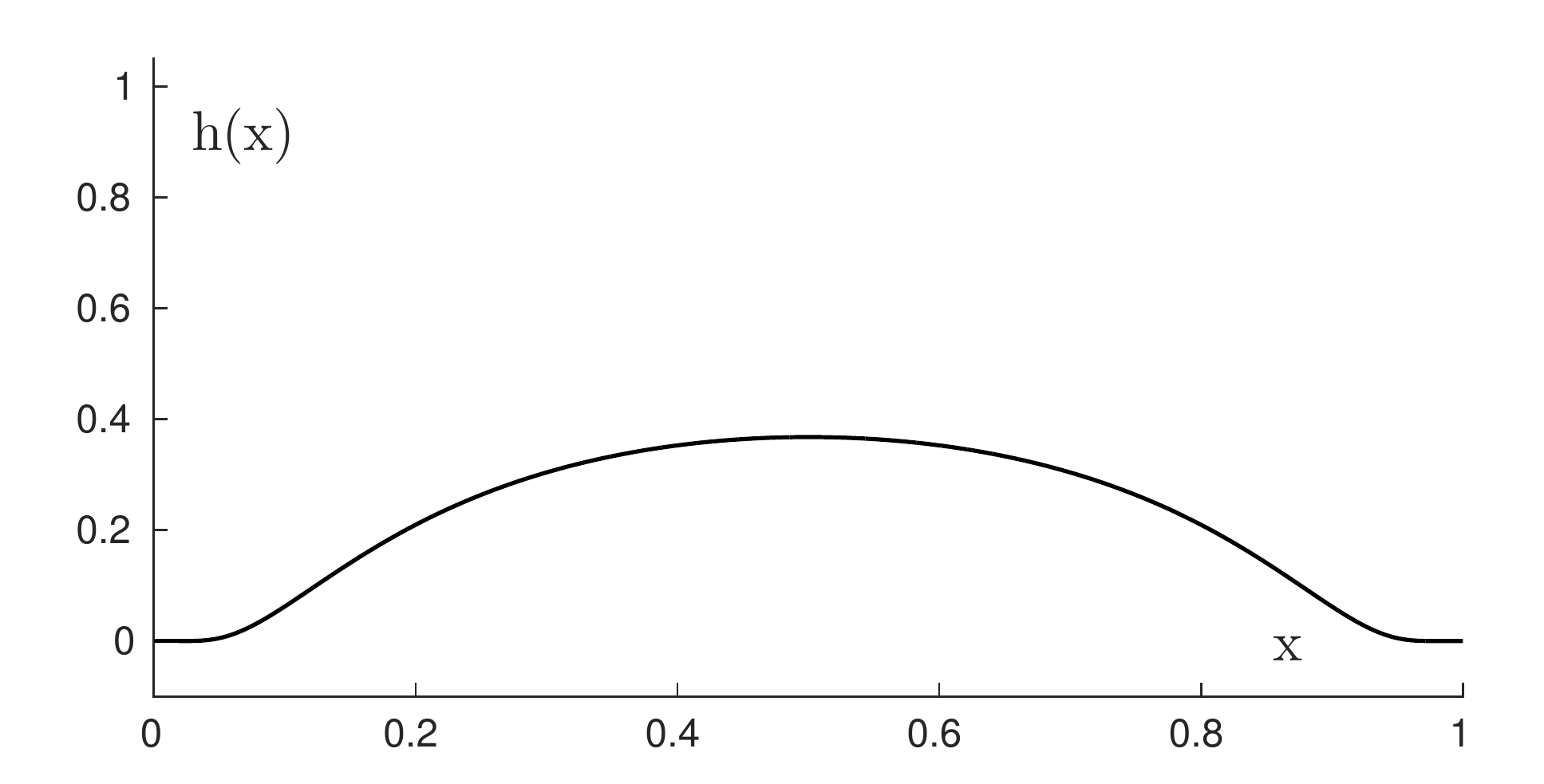}
\end{minipage}
\begin{minipage}[h!]{.46\linewidth}
\includegraphics[width=\linewidth]{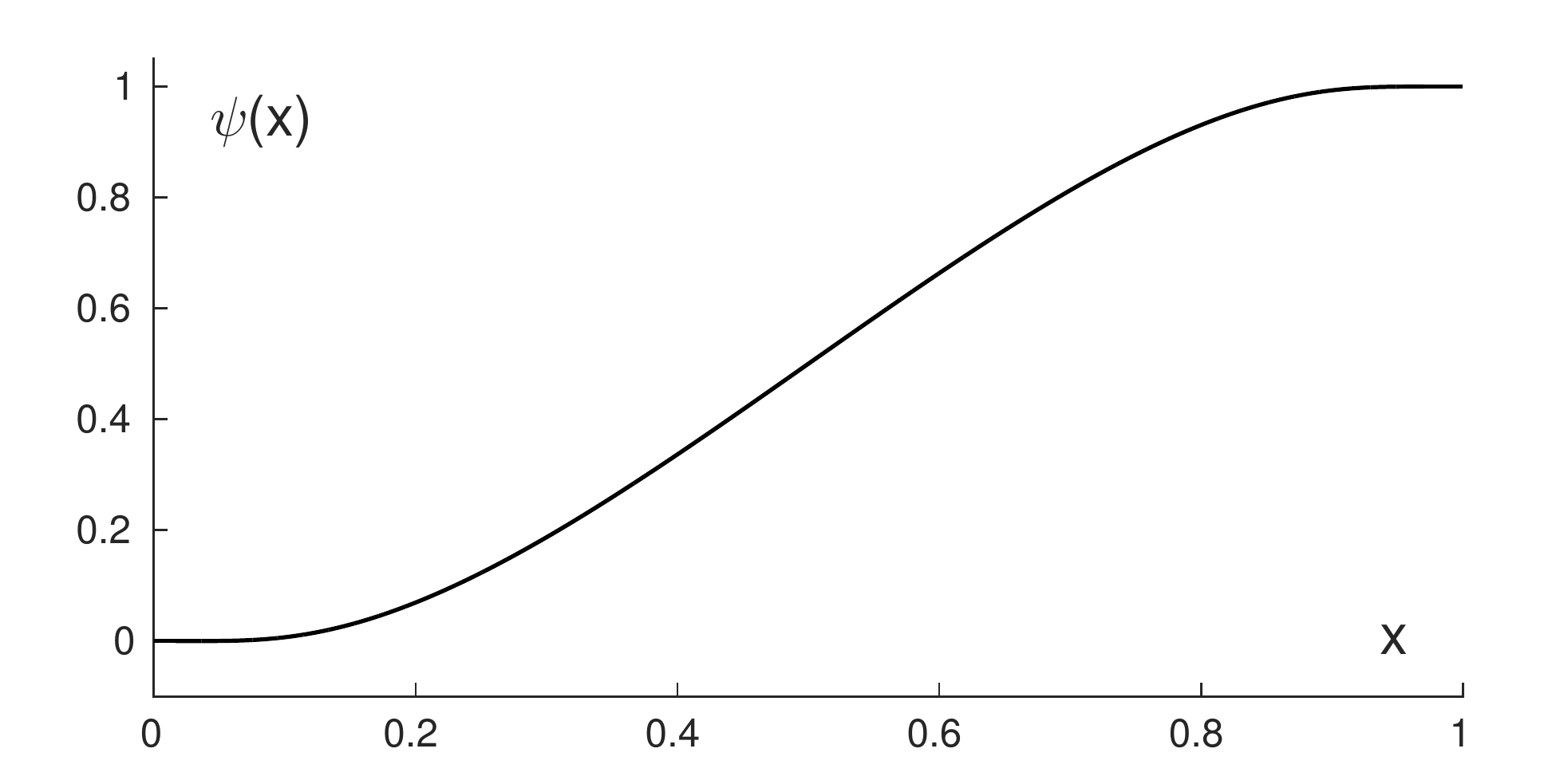}
\end{minipage}

Given such $\psi$, the map $\Psi:\IR^d\to\IR^d$ with 
$\Psi(x)=(\psi(x_1),\hdots,\psi(x_d))^{\top}$ is a diffeomorphism on $(0,1)^d$ 
with inverse $\Psi^{-1}(x)=(\psi^{-1}(x_1),\hdots,\psi^{-1}(x_d))^{\top}$ 
and $|D\Psi(x)|\overset{\psi'\geq 0}{=}\det D\Psi(x)=\prod\limits_{i=1}^{d}\psi'(x_i)$.

If $A_n$ is any linear quadrature formula for integration on the unit cube
with nodes $x^{(j)}\in[0,1]^d$ and weights $a_j\in\IR$, 
where $j=1,\hdots,n$, we define the transformed quadrature formula 
$\widetilde{A}_n$ by choosing the nodes and weights to be
\[
\tilde{x}^{(j)}=\Psi(x^{(j)})$ \text{\ \ \ \ and \ \ \ \ } 
$\tilde{a_j}=a_j\cdot|D\Psi(x^{(j)})|.
\]
Thus, $\widetilde{Q}_S^v$ for $v\in\IR^d$ and invertible $S\in\IR^{d\times d}$ takes the form
\[
\widetilde{Q}_S^v(f)=\frac{1}{\abs{\det S}}
\sum\limits_{m\in\IZ^d} f\left(\Psi\left(S^{-\top}(m+v)\right)\right)
\cdot \left|D\Psi\left(S^{-\top}(m+v)\right)\right|
\]
for any input function $f:[0,1]^d\to \IR$. Note that
$\left|D\Psi\left(S^{-\top}(m+v)\right)\right|$ is zero for any index
$m\in\IZ^d$ with $S^{-\top}(m+v)\not\in [0,1]^d$.

\medskip 
\noindent
{\bf Algorithms.} 
For any $n\in\IN$ we consider the transformed versions
$\widetilde{Q}_{n^{1/d}B}$ and $\widetilde{Q}_{n^{1/d}\hat{U}B}^V$ of the algorithms
$Q_{n^{1/d}B}$ and $Q_{n^{1/d}\hat{U}B}^V$ from Section \ref{frolovsrulesection}
and Section \ref{randomshiftsection}.

\medskip

These algorithms are well defined for any input function $f:[0,1]^d\to\IR$.
The information costs of $\widetilde{Q}_{n^{1/d}B}$ and $\widetilde{Q}_{n^{1/d}\hat{U}B}^V$ are of order $n$.
By Lemma \ref{anlemma}, they are at most
$2\cdot\left(\Vert B\Vert_1+1\right)^d\cdot n$.

\begin{prop}
\label{MSchlangeunbiased}
The method $\widetilde{Q}_{n^{1/d}\hat{U}B}^V$ is well-defined and unbiased on $L^1([0,1]^d)$.
\end{prop}

\begin{proof}
Let $f\in L^1([0,1]^d)$. By the Change of Variables Theorem,
$f_0=f\circ \Psi \cdot \left|D\Psi\right|$ is also integrable on $[0,1]^d$
and satisfies
\[
I_d(f)=I_d(f_0)\quad \text{and}\quad \widetilde{Q}_{n^{1/d}\hat{u}B}^v(f) = Q_{n^{1/d}\hat{u}B}^v(f_0)
\]
for any realization $\widetilde{Q}_{n^{1/d}\hat{u}B}^v$ of $\widetilde{Q}_{n^{1/d}\hat{U}B}^V$.
This yields
\[
\IE \left(\widetilde{Q}_{n^{1/d}\hat{U}B}^V(f)\right) = \IE \left( Q_{n^{1/d}\hat{U}B}^V(f_0)\right) = I_d(f_0) = I_d(f)
\]
by Proposition~\ref{Munbiased}.
\end{proof}

Most notably, $\widetilde{Q}_{n^{1/d}\hat{U}B}^V$ satisfies the following error bounds on $\Hmixunit$.

\begin{thm}
\label{MSchlangebounds}
There is some $c>0$ such that for every $n\geq 2$ and $f\in\Hmixunit$
\[\begin{split}
\braces{\IE\abs{\widetilde{Q}_{n^{1/d}\hat{U}B}^V(f)-I_d(f)}^2}^{1/2}
&\leq c\, n^{-r-1/2}\,\mixnormunit{f}
\quad \text{and}\\
\sup\limits_{(u,v)\in[1,2^{1/d}]^d\times[0,1]^d} \abs{\widetilde{Q}_{n^{1/d}\hat{U}B}^V(f)-I_d(f)}
&\leq\, c \,  n^{-r}\, (\log n)^\frac{d-1}{2}
\, \mixnormunit{f}
.\end{split}\]
\end{thm}

\begin{proof}
Recall that
$\widetilde{Q}_{n^{1/d}\hat{u}B}^v(f) = Q_{n^{1/d}\hat{u}B}^v(f_0)$ and $I_d(f)=I_d(f_0)$ 
for any function $f\in L^1(\IR^d)$ and $f_0=f\circ \Psi \cdot \left|D\Psi\right|$.

Since $\psi'(x)=0$ for $x\not\in(0,1)$, we have 
$\diff^\alpha f_0|_{\partial[0,1]^d}=0$ for each $\alpha\in\{0,\dots,r\}^d$ 
and hence $f_0\in \Hmixunitcomp$ for any $f\in\Hmixunit$.\\
That implies the estimates
\[\begin{split}
\braces{\IE \abs{\widetilde{Q}_{n^{1/d}\hat{U}B}^V(f)-I_d(f)}^2}^{1/2}
&= \braces{\IE \abs{ Q_{n^{1/d}\hat{u}B}^v(f_0)-I_d(f_0)}^2}^{1/2}\\
&\leq c \cdot n^{-r-1/2}\cdot \mixnormunit{f_0}, \\
\sup\limits_{(u,v)\in[1,2^{1/d}]^d\times[0,1]^d} \abs{\widetilde{Q}_{n^{1/d}\hat{u}B}^v(f)-I_d(f)}
&= \sup\limits_{(u,v)\in[1,2^{1/d}]^d\times[0,1]^d} \abs{ Q_{n^{1/d}\hat{u}B}^v(f_0)-I_d(f_0)}\\
&\leq c \cdot n^{-r} (\log n)^\frac{d-1}{2} \cdot \mixnormunit{f_0}
,\end{split}\]
if $c>0$ is the maximum of the constants of 
Theorem~\ref{mixthm2} and Corollary~\ref{mixthmworstcase2}.
That proves the statement,
since there is some $c_0>0$ such that every $f\in\Hmixunit$ 
satisfies $\mixnormunit{f_0}\leq c_0\, \mixnormunit{f}$. 
This can be proven as follows.

The partial derivatives of $f_0$ take the form
\[
D^\alpha f_0(x)=\frac{\partial^{\abs{\alpha}}}{\partial 
x_1^{\alpha_1}\cdots\partial x_d^{\alpha_d}}\ f(\Psi(x))
\cdot\prod\limits_{i=1}^{d}\psi'(x_i) 
= \sum\limits_{\beta_1,\hdots,\beta_d=0}^{\alpha_1,\hdots,\alpha_d} 
D^\beta f(\Psi(x))\cdot S_{\alpha,\beta}(x)
\]
for $\alpha\in\{0,1,\hdots,r\}^d$, where $S_{\alpha,\beta}(x)$ is a finite 
sum of finite products of terms $\psi^{(j)}(x_i)$ with $i\in\{1,\hdots,d\}, 
j\in\{1,\hdots,rd+1\}$ and does not depend on $f$. It is therefore continuous 
and bounded by some $c_{\alpha,\beta}>0$.

A special case of Cauchy's inequality says that
the square of a sum is at most the sum of the squares
times the number of addends.

Using these facts, we get
\[\begin{split}
\lnormunit{D^\alpha f_0}^2
&\leq \left(\sum\limits_{\beta_1,\hdots,\beta_d=0}^{\alpha_1,\hdots,\alpha_d} 
\lnormunit{(D^\beta f \circ \Psi) \cdot S_{\alpha,\beta}}\right)^2\\
&\leq \left(\sum\limits_{\beta_1,\hdots,\beta_d=0}^{\alpha_1,\hdots,\alpha_d} 
c_{\alpha,\beta} \cdot\lnormunit{D^\beta f \circ \Psi}\right)^2\\
&\leq (r+1)^d\sum\limits_{\beta_1,\hdots,\beta_d=0}^{\alpha_1,
\hdots,\alpha_d} c_{\alpha,\beta}^2 \cdot\lnormunit{D^\beta f \circ \Psi}^2\\
&= (r+1)^d\sum\limits_{\beta_1,\hdots,\beta_d=0}^{\alpha_1,\hdots,\alpha_d} 
c_{\alpha,\beta}^2 \int_{(0,1)^d} |D^\beta f (\Psi (x))|^2 \,\mathrm{d}x\\
&= (r+1)^d\sum\limits_{\beta_1,\hdots,\beta_d=0}^{\alpha_1,\hdots,\alpha_d} 
c_{\alpha,\beta}^2 \int_{\Psi\left((0,1)^d\right)} |D^\beta f (\Psi (\Psi^{-1}(x))|^2 
\cdot |D\Psi^{-1}(x)| \, \mathrm{d}x\\
&\leq (r+1)^d\sup\limits_{x\in(0,1)^d} |D\Psi^{-1}(x)| 
\sum\limits_{\beta_1,\hdots,\beta_d=0}^{\alpha_1,\hdots,\alpha_d} c_{\alpha,\beta}^2 
\cdot \lnormunit{D^\beta f}^2
\end{split}\]
\[\begin{split}
&\leq c_\alpha \cdot \mixnormunit{f}^2\
,\end{split}\]
for some $c_\alpha>0$ and
\[
\mixnormunit{f_0}^2
=\sum\limits_{\alpha\in\{0,1,\hdots,r\}^d}\lnormunit{D^\alpha f_0}^2
\leq \tilde{c} \, \mixnormunit{f}^2
,\]
if $\tilde{c}$ is the sum of the constants $c_\alpha$ for $\alpha\in\set{0,1,\dots,r}^d$.
\end{proof}

The corresponding error bounds for $\widetilde{Q}_{n^{1/d}\hat{U}B}^V$ on $\Hisounit$ are proven in the exact same manner.

\begin{thm}
\label{MSchlangebounds2}
There is some $c>0$ such that for every $n\in\IN$ and $f\in\Hisounit$
\[\begin{split}
\braces{\IE\abs{\widetilde{Q}_{n^{1/d}\hat{U}B}^V(f)-I_d(f)}^2}^{1/2}
&\leq c\, n^{-s/d-1/2}\,\isonormunit{f}
\quad \text{and}\\
\sup\limits_{(u,v)\in[1,2^{1/d}]^d\times[0,1]^d} \abs{\widetilde{Q}_{n^{1/d}\hat{u}B}^v(f)-I_d(f)}
&\leq\, c \,  n^{-s/d}
\, \isonormunit{f}
.\end{split}\]
\end{thm}

The optimality of this order of convergence of the expected error on $\Hisounit$
was already stated by N.\,S.\,Bakhvalov in 1962, see \cite{bakhvalov}.
A proof can be found in \cite{ln}.
The optimality of the upper bound of Theorem~\ref{mixthm2}
for arbitrary dimensions can be derived from Bakhvalov's result
for the one-dimensional case.\\

Since the transformed version $\widetilde{Q}_{n^{1/d}B}$ of Frolov's deterministic algorithm is a particular realization
of the method $\widetilde{Q}_{n^{1/d}\hat{U}B}^V$, Theorem~\ref{MSchlangebounds} and \ref{MSchlangebounds2} imply the
following error bounds.

\begin{cor}
There is some $c>0$ such that for any $n\geq 2$ and $f\in\Hmixunit$
\[
\abs{\widetilde{Q}_{n^{1/d}B}(f)-I_d(f)} \leq\, c \,  n^{-r} 
\, (\log n)^\frac{d-1}{2} \, \mixnormunit{f}
\]
and for any $n\in\IN$ and $f\in\Hisounit$
\[
\abs{\widetilde{Q}_{n^{1/d}B}(f)-I_d(f)} \leq\, c \,  n^{-s/d} 
\, \isonormunit{f}
.\]
\end{cor}

It is also not hard to see, that the error bounds for $Q_{n^{1/d}\hat{U}B}^v$ from Theorem~\ref{mixthmworstcase},
\ref{isothmworstcase}, \ref{mixthm} and \ref{isothm}
on the classes $\Hmixunitcomp$ and $\Hisounitcomp$ are inherited by
the method $\widetilde{Q}_{n^{1/d}\hat{U}B}^v$ on the classes $\Hmixunit$ and $\Hisounit$
in the same way.\\

To sum up, both the expected error and the worst case error
of the method $\widetilde{Q}_{n^{1/d}\hat{U}B}^V$ have an optimal
rate of convergence
on both $\Hmixunit$ and $\Hisounit$.
In addition, the method is unbiased.
It is also worth stressing that the algorithm is universal:
It does not depend on the smoothness $r$ or $s$ of the input function in any way
and hence no prior knowledge of it is needed to run $\widetilde{Q}_{n^{1/d}\hat{U}B}^V$.
Nonetheless, the convergence rate of its error perfectly adjusts to that smoothness.
The same is valid for the algorithms $\widetilde{Q}_{n^{1/d}B}$ and $\widetilde{Q}_{n^{1/d}\hat{U}B}^v$.

\end{document}